\documentclass[11pt]{article}
\usepackage{amssymb}
\usepackage{epsf,latexsym,amssymb,amsmath,amsthm}
\usepackage{amsfonts}
\usepackage{amsthm}
\usepackage{amsmath}
\usepackage{amssymb}
\usepackage{graphicx}
\usepackage{epsfig}
\usepackage{color}

\usepackage[left=1.4cm,bottom=1.6cm,right=1.4cm,top=1.2cm]{geometry}

\providecommand{\U}[1]{\protect\rule{.1in}{.1in}}

\def\div{\mathop{\rm div}\nolimits}

\def\u{\underline}

\def\1{{1\hskip-0.25em{\rm l}}}

\def\p{\partial}
\def\ep{\varepsilon}

\def\div{{\rm div}}

\def\u2{{u^\ep \over \ep^2 }}
\def\u3{{\displaystyle {\bar u}^\ep \over \ep^2 }}

\def\div{{\rm div}}

\def\ep{\varepsilon}

\def\p{\partial}

\def\u2{{u^\ep \over \ep^2 }}
\def\u3{{\displaystyle {\bar u}^\ep \over \ep^2 }}

\def\dsp{\displaystyle}

\def\Cc{{\cal C}}

\newtheorem{theorem}{Theorem}

\newtheorem{corollary}[theorem]{Corollary}

\newtheorem{lemma}[theorem]{Lemma}

\newtheorem{proposition}[theorem]{Proposition}
\newtheorem{remark}[theorem]{Remark}

\begin{document}

\title{Asymptotic analysis of the Poisson-Boltzmann equation describing electrokinetics in  porous media\thanks{The research  was partially supported by the GNR MOMAS CNRS-2439
(Mod\'elisation Math\'ematique et Simulations num\'eriques
 li\'ees aux probl\`emes de gestion des d\'echets nucl\'eaires) (PACEN/CNRS, ANDRA,
 BRGM, CEA, EDF, IRSN) and GNR PARIS (Propri\'et\'es des actinides et des radionucleids aux interfaces et aux solutions). The authors would like to thank O. Bernard, V. Marry, B. Rotenberg et P. Turq  from the Mod\'elisation et Dynamique Multi-\'echelles team from the laboratory Physicochimie des Electrolytes, Collo\"ides et Sciences Analytiques (PECSA), UMR CNRS 7195, Universit\'e P. et M. Curie, for the helpful discussions. G. A. is a member of the DEFI project at INRIA Saclay Ile-de-France.}}

\author{Gr\'egoire Allaire$^1$
 \and Jean-Fran\c{c}ois Dufr\^eche$^2$
 \and Andro Mikeli\'c $^3$ \and      Andrey Piatnitski $^4$
 \\  $^1$ CMAP, Ecole Polytechnique,
F-91128 Palaiseau, France \\
[2pt] $^2$ Universit\'e de Montpellier 2, Laboratoire Mod\'elisation \\ M\'esoscopique et Chimie Th\'eorique (LMCT),\\
Institut de Chimie S\'eparative de Marcoule ICSM
UMR 5257,\\ CEA / CNRS / Universit\'e de Montpellier 2 / ENSCM\\
Centre de Marcoule, B\^at. 426, BP 17171,
30207 Bagnols sur C\`eze Cedex, France \\
[2pt] $^3$ Universit\'e de Lyon, CNRS UMR 5208,\\
  Universit\'e Lyon 1, Institut Camille Jordan, \\   43, blvd. du 11 novembre 1918,
 69622 Villeurbanne Cedex, France\\
[2pt] $^4$ Narvik University College, Norway and  Lebedev Physical Institute, Moscow, Russia  }
\maketitle

\begin{abstract}
We consider the Poisson-Boltzmann equation in a periodic cell, representative of a porous
medium. It is a model for the electrostatic distribution of $N$ chemical species diluted
in a liquid at rest, occupying the pore space with charged solid boundaries.
We study the asymptotic behavior of its solution depending on a parameter $\beta$
which is the square of the ratio between a characteristic pore length and the Debye length.
For small $\beta$ we identify the limit problem which is still a nonlinear Poisson equation
involving only one species with maximal valence, opposite to the average of the given
surface charge density. This result justifies the {\it Donnan effect}, observing that the ions for which the charge is the one of the solid phase are expelled from the pores. For large $\beta$ we prove that the solution behaves like a
boundary layer near the pore walls and is constant far away in the bulk. Our analysis
is valid for Neumann boundary conditions (namely for imposed surface charge densities) and establishes rigorously  that solid interfaces are uncoupled from the bulk fluid, so that the
simplified additive theories, such as the one of the popular Derjaguin, Landau, Verwey and Overbeek
(DLVO) approach, can be used.
We show that the asymptotic behavior is completely different in the case of Dirichlet
boundary conditions (namely for imposed surface potential).
\end{abstract}

\bigskip


\bigskip


\bigskip


\section{Introduction}
\label{preintro}
Originally proposed at the dawn of the XX$^\mathrm{th}$ century by Gouy and Chapman \cite{Gouy,Chapman},
the Poisson-Boltzmann (PB) equation is still the corner stone of most of the theoretical descriptions
of electrokinetic phenomena. Many works emphasized the limitation of such a model in
the last decades, though. The ions are only represented by their charge, they do not have any volume,
the correlation are neglected. The molecular nature of the solvent and further
specific forces (such as the London dispersion) are completely
ignored \cite{Dispersion,Clayjardat}. Thus the domain of validity appears to be relatively narrow,
typically in the regime of dilute simple (most of the time monovalent) electrolytes.
Nevertheless, because of its simplicity, most of the theories of equilibrium and transport
in charged diphasic media are still direct generalization of the PB approach. For example, geological
media (such as clays) \cite{Moyne,Trizac}, electrochemistry \cite{Bard,KBA:05}, and colloidal
physics \cite{Lyklema} are still based on the original concepts described by the Poisson-Boltzmann equation.

The success of such an approach is due to several aspects. It justifies the popular Derjaguin, Landau, Verwey and Overbeek
(DLVO) theory
 \cite{Verwey} that explains the stability of charged suspensions.
 In the case of charged porous media,
the Poisson-Boltzmann approach is also particularly significant because it yields the equilibrium electrostatic
properties of the materials and it can be easily coupled to
further equations in order to provide a global model of the system.
Indeed, for the transport properties, the Poisson-Boltzmann equation can be extended in order
to give the Poisson-Nernst-Planck (PNP) formalism which describes
non-equilibrium processes in complex systems \cite{Zheng}. For example, in
the case of clays, the description of electrokinetic processes in the large pores (meso and macroporosities)
can be performed thanks to the PB equation \cite{Alder}. The molecular nature of the system are found to
be important only for micropores (typically for distances less than 2~nm) \cite{Marry}.

In porous media, the PB exhibit two different regimes, depending on the value of the salt concentration or the pore
sizes on the system.
\begin{itemize}
\item If the pore size $L$ is much larger than the Debye length $\lambda_D$ of the electrolyte, the solid charge
is screened by the microscopic ions. Thus, the local charge density is globally zero, but at the interface.
Because of the relatively small value $\lambda_D$, this case corresponds to numerous applications.
The solid interfaces are uncoupled so that the DLVO approach is valid. Far away from the interface,
the coulombic forces can be modelled by effective parameters, such as the effective charge \cite{Eff}
or the zeta potential.
\item Conversely, if the pore size $L$ is much smaller than $\lambda_D$, the charge of the solid surface is not screened. It means that
the resulting electrostatic force is important anywhere in the material. Ions for which the charge is
the same than the one of the solid phase are expelled from the material (Donnan effect) \cite{Clayjardat}.
The electro-osmotic flux becomes especially important. This case is significant because it corresponds to
nanoporosities at low salt concentration.
\end{itemize}
The two asymptotic limits can be taken into account thanks to the coupling parameter $\beta=(L/\lambda_D)^2$.
The large pore size domain (large $\beta$) corresponds to most of the porous systems. Nevertheless, many
porous materials exhibit microscopic pores for which the opposite limit (small $\beta$) is
relevant. For example, montmorillonite clays have different porosities, and the smallest ones, which are
obtained at very low hydration, are even of the order of molecular distances.

The goal of the present paper is to give a rigorous mathematical analysis of these two 
opposite asymptotic limits. The paper is organized as follows. Section \ref{introd} introduces 
the model and defines the relevant reduced units. Section \ref{PerioPB} describes precisely 
the geometry of the porous cell and discusses the issue of the existence and uniqueness of 
the solution to PB equation. Section \ref{small} studies the limit case of very small pores, 
i.e., when $\beta$ goes to zero. Section \ref{Lb} is concerned with the opposite situation 
of very large pores, when $\beta$ goes to infinity. Eventually Section \ref{dirichlet} 
investigates the case of Dirichlet boundary conditions (namely for imposed surface potential) 
instead of Neumann boundary condition (namely imposed surface charge). 
A brief description of our main results is given in the next section after introducing 
the necessary notations.

\section{The model and our main results}
\label{introd}

We consider the Poisson-Boltzmann system which describes the electrostatic distribution
of $N$ chemical species diluted in a liquid at rest, occupying a porous medium with
charged solid boundaries.
The electrostatic potential $\Psi^*$ is calculated from the Poisson equation
\begin{gather}
\label{poisson}
\mathcal{E} \Delta \Psi^* = - e \sum^N_{j=1} z_j n_j^* \quad \mbox{ in the bulk}  ,
\end{gather}
where $\mathcal{E} = \mathcal{E}_0 \mathcal{E}_r$ is the dielectric constant
of the solvent, $e$ is the electron charge and $n_j^*$, $1\leq j\leq N$, are
the species concentrations. Since the pore walls are charged, the corresponding
boundary condition is of Neumann type
\begin{gather}
\label{poissonbc}
\mathcal{E} \nabla \Psi^* \cdot {\bf n} = -\sigma^* \quad \mbox{ on the surface} ,
\end{gather}
where $\sigma^*$ is a given surface charge and ${\bf n}$ is the unit exterior normal.

At equilibrium the species concentrations $n_j^*$ are given by the Boltzmann distribution
which corresponds to a balance between the chemical potential and the electrical
field
\begin{gather}
\label{boltzmann}
\nabla (k_B T \ln n_j^*) = - \nabla (z_j e \Psi^*) .
\end{gather}
where $z_j$ is the valence of the $j$-th species, $k_B$ the Boltzmann constant
and $T$ the temperature. It follows from (\ref{boltzmann}) that there exist
positive constants $n_j^*(\infty)$ (called infinite dilution concentrations)
such that
\begin{equation}
\label{boltzmann2}
n_j^* =  n_j^*(\infty) \exp \left\{ - \frac{z_j e \Psi^*}{k_B T}  \right\} .
\end{equation}
The Poisson-Boltzmann system is the combination of (\ref{poisson}) and (\ref{boltzmann2}),
together with the boundary condition (\ref{poissonbc}).

In order to make an asymptotic analysis of the Poisson-Boltzmann system, we first
adimensionalize equations (\ref{poisson}), (\ref{poissonbc}), (\ref{boltzmann2}).
We denote by $L$ the characteristic pore size and by $n_c$ the characteristic concentration.
We introduce the Debye length defined by
$$
\lambda_D = \sqrt{\frac{\mathcal{E} k_B T}{ e^2 n_c}},
$$
and define a parameter
\begin{equation}
\label{def.beta}
\beta = (\frac{L}{\lambda_D})^2 .
\end{equation}
The parameter $\beta$ is the fundamental physical characteristic which drives the transport properties of an electrolyte solution in a porous media.
For large $\beta$ the electrical potential is concentrated in a diffuse layer next to the liquid/solid interface. Co-ions, for which the charge is the same as the one of the solid phase are able to go everywhere in the pores because the repelling electrostatic force of the solid phase is screened by the counterions. The electrostatic phenomena are mainly surfacic, and the interfaces are globally independent.  For small $\beta$, co-ions do not have access to the very small pores (Donnan effect). The local electroneutrality condition is not valid anymore and the electric fields of the solid interfaces are coupled.

Furthermore, we define a characteristic surface charge density $\sigma_c$ by
$$
\sigma_c = \frac{\mathcal{E} k_B T}{e L} ,
$$
and adimensionalized quantities
$$
\sigma = \frac{\sigma^*}{\sigma_c} , \quad
\Psi = \frac{e\Psi^*}{k_B T} , \quad n_j = \frac{n_j^*}{n_c}  ,
\quad n_j^0(\infty) = \frac{n_j^*(\infty)}{n_c} .
$$
Rescaling the space variable $y=x/L$,
this yields the adimensionalized Poisson-Boltzmann system
\begin{equation}
\label{PB1}
\Delta_y \Psi = - \beta \sum^N_{j=1} z_j n_j(\Psi) \quad \mbox{ with } \
n_j(\Psi) =  n_j^0(\infty) e^{ - z_j \Psi } \ \mbox{ in the bulk, }
\end{equation}
and
\begin{equation}
\label{PB2}
\nabla_y \Psi \cdot {\bf n} = -\sigma \quad \mbox{ on the surface.}
\end{equation}
The goal of the present paper is to study the limit of equations
(\ref{PB1}) and (\ref{PB2}) when the parameter $\beta$ is either
very small or very large. 
Section \ref{PBE} gives a precise mathematical framework for
the Poisson-Boltzmann system.
To discuss our results we sort the valences by increasing order
and we assume that there are both anions and cations
$$
z_1 < z_2 < ... < z_N \quad \mbox{ and } \quad z_1<0<z_N .
$$
Section \ref{small} is devoted to the asymptotic analysis of (\ref{PB1})-(\ref{PB2})
when $\beta$ goes to zero. This case corresponds to very small pores,
$L<<\lambda_D$. In view of the definition of the Debye length
$\lambda_D$, a small value of $\beta$ corresponds also to a
small characteristic concentration $n_c$. The asymptotic
regime depends on the sign of the averaged charge $\int_S \sigma \, dS$.
If it is negative (which means that the surface is positively charged),
then Theorem \ref{th1} states that only the anion with the most negative
valence ($z_1$) is important and that the potential behaves as
$$
\Psi \approx \frac{\log\beta}{z_1} + \varphi_0 ,
$$
where $\varphi_0$ is the solution of the reduced system, involving
only the species $1$,
$$
\left\{ \begin{array}{l}
\dsp \Delta \varphi_0 = - z_1 n^0_1(\infty) e^{-z_1 \varphi_0} \ \mbox{ in the bulk, }  \\
\dsp \nabla \varphi_0 \cdot {\bf n} = -\sigma \ \mbox{ on  the surface.}
\end{array} \right.
$$
Note that the constant $\log\beta /z_1$ is going to $+\infty$ since $z_1<0$:
it is a manifestation of the singularly perturbed character of the asymptotic
analysis.
As a consequence, the cation concentrations goes to zero while the ion
concentrations blow up as $n_j= O(\beta^{-z_j/z_1})$ and $n_1>>n_j$ for
$j\neq1$. Of course, a symmetric behavior (involving only the species
with most positive valence $z_N$) holds true when the averaged charge
is positive. On the other hand, when the averaged charge is zero, then
the limit problem is much simpler (there is no singular perturbation)
and it is given by Theorem \ref{th2}.

Section \ref{Lb} is concerned with the opposite situation when
$\beta$ goes to infinity. This case corresponds to very large pores,
$L>>\lambda_D$ or to large values of the characteristic concentration $n_c$.
Such a situation is well understood in the physical and mathematical literature:
it gives rise to a boundary layer (the so-called Debye's layer) close to the surface.
For example, it has been analyzed in \cite{BCEJ:97}, \cite{PJ:97} but the
analysis is restricted to 1-d or similar simplified geometry. Our main
contribution, Theorem \ref{estlargebet}, derives this boundary layer
in a general geometric setting and gives a rigorous error estimate.
Locally, close to the surface, the potential behaves as
$$
\Psi(y) \approx \frac{-\sigma}{\sqrt{\beta \sum^N_{j=1} z_j^2 n_j^0(\infty)} }
\exp \left\{ - d(y) \sqrt{\beta \sum^N_{j=1} z_j^2 n_j^0(\infty)} \right\}  ,
$$
where $d(y)$ is the distance between the point $y$ and the surface.
Away from the surface, the concentrations $n_j$ are constant and
satisfy the so-called bulk electroneutrality condition.
Some numerical simulations for varying $\beta$ can be found in 
our other work \cite{ABDMP}.

In the literature, one may find instead of the Neumann boundary
condition (\ref{poissonbc}) a Dirichlet one
\begin{equation}
\label{PB3}
\Psi = \zeta \quad \mbox{ on the surface,}
\end{equation}
where $\zeta$ is an imposed potential.
Eventually Section \ref{dirichlet} investigates the case when
the Neumann boundary condition (\ref{poissonbc}) is replaced
by the Dirichlet boundary condition (\ref{PB3}). Lemma \ref{Dir.small}
gives the asymptotic behavior for small $\beta$: the analysis is
quite obvious since there is no more singular perturbation.
The potential now behaves as
$$
\Psi \approx \zeta + \beta\Psi_1 ,
$$
where $\Psi_1$ is the solution of the corrector problem
$$
\left\{ \begin{array}{l}
\dsp \Delta \Psi_1 = -\sum^N_{j=1} z_j n_j^0(\infty) e^{ - z_j \zeta } \ \mbox{ in the bulk, }  \\
\dsp \Psi_1 = 0 \ \mbox{ on  the surface.}
\end{array} \right.
$$
Theorem \ref{estlargebetZ} gives the asymptotic behavior for large
$\beta$. It is again a boundary layer but with a totally different
profile. More precisely we establish
$$
\Psi(y) \approx \Psi_{0,\zeta} \Big( \sqrt\beta d(y) \Big)
$$
where $d(y)$ is the distance between the point $y$ and the surface
and $\Psi_{0,\zeta}$ is the solution of the nonlinear ordinary differential
equation (\ref{Cauchy}), solution which, starting from the boundary value
$\zeta$ on the surface, is exponentially decaying at infinity.

\section{Geometry and existence theory for the Poisson-Boltzmann equation}
\label{PerioPB}

In this section we define the geometry and recall a classical result on the
existence and uniqueness of the solution of the Poisson-Boltzmann equation.
Furthermore we establish $L^\infty$-bounds for the solution.

\subsection{Geometry of the porous cell}
\label{Geo}

To simplify the presentation we consider a sample of a porous medium
which occupies the periodic unit cell $Y = (0,1)^d$ (identified with
the unit torus $\mathbb T^d$). The space dimension is typically $d=2,3$.
The pore space $Y_F$, containing the electrolyte, is a $1$-periodic smooth
connected open subset of $Y$. More precisely, we consider a smooth partition
$Y=Y_S\cup Y_F$ where $Y_S$ is the solid part and $Y_F$ is the fluid part.
The liquid/solid interface is $S=\partial Y_S \setminus \p Y$.
It is known that for a general $Y_F$ with non-empty boundary $S$,
the distance function $d$, defined by $d(y)= $ dist $(y, S)$, $y\in Y_F$,
is uniformly Lipschitz continuous. Let $Y_F^\mu = \{ y\in Y_F \ : \, d(y) < \mu \}$.
If we assume $S$ to be smooth of class $C^3$, then, for sufficiently small $\mu$,
the distance function
has the same regularity, $d\in C^3 (Y_F^\mu)$. Furthermore, by the Collar Neighborhood
Theorem (see e.g. \cite{spivak}), there exists a tubular neighborhoods $U$,
$S \subset U$, isomorphic to $S \times I$, for $I\subset \mathbb{R}$ an open interval.

For $y_0\in S$, let $\mathbf{n} (y_0)$ and $\mathbf{T} (y_0)$ denote respectively
the unit normal to $S$ at $y_0$ (exterior to $Y_F$) and the tangent hyperplane to $S$ at $y_0$.
Without loosing generality, we can suppose that in some neighborhood
$\mathcal{N} = \mathcal{N}(y_0)$ of $y_0$, $S$ is then given by $y_d = \gamma(y')$,
where $y'= (y_1,\dots, y_{d-1} )$, $\gamma\in C^3 (\mathbf{T} (y_0) \cap \mathcal{N})$
and $\nabla_{y'} \gamma (y'_0 )=0$. The unit normal on $S$, at the point $y= (y', \gamma(y'))$
is
$$
\mathbf{n} (y)=\frac{(\nabla_{y'} \gamma (y'), -1)}{\sqrt{1+| \nabla_{y'} \gamma (y')|^2 }} .
$$
Here $Y_S$ corresponds to $y_d > \gamma (y')$.

Following \cite{GT}, page 355, for each point $y \in Y_F^\mu $, there exists a unique point
$z=z(y) \in S$ such that $|y -z| = d(y)$. The points $y$ and its closest point $z\in S$ are
related by
\begin{equation}\label{close}
y = z - \mathbf{n} (z) d .
\end{equation}
By the Inverse Mapping Theorem applied in a neighborhood of $y_0\in S$
($S$ can be covered with a finite number of such neighbourhoods),
(\ref{close}) defines uniquely a principal coordinate system
$(q',q_d)$ with $q'=z'$ and $q_d=d(y)$ which are $C^2$ functions of $y$.
The first coordinates $q'$ are the tangential coordinates to $S$ while
the last coordinate $q_d$ is the signed distance to $S$ taken with the sign $+$
in the interior of $Y_F$ and with $-$ in its exterior. We denote by $q_d=0$
the local representation of the boundary $S$ in the neighborhood of $y_0$.
Of course, the distance function $d\in C^3$ satisfies
$\nabla d(y) = -\mathbf{n} (q'(y))$ and $|\nabla d(y)| =1$.
Furthermore, the tangential coordinates $q'= (q_1,\dots, q_{d-1})$ can be
chosen in such a way that the Hessian matrix $[D^2 \gamma (q' (y_0))]$
is diagonal at $y_0$.
For more details we refer to \cite{GT}, pages 354-356.

\subsection{Mathematical framework for the Poisson-Boltzmann equation}
\label{PBE}

We start from the adimensionalized Poisson-Boltzmann system (\ref{PB1})
and (\ref{PB2}) (where we drop the index $y$ for the spatial derivatives)
that we complement with periodic boundary conditions on the outer boundary
of the unit cell $Y$. More precisely, the potential $\Psi$ is a solution
of the Poisson-Boltzmann equation
\begin{equation}
\label{BP1}
\left\{ \begin{array}{l}
        \dsp - \Delta \Psi + \beta  \Phi ( \Psi  ) =0  \ \mbox{ in } \ Y_F , \\
        \dsp   \nabla \Psi \cdot \mathbf{n} = -\sigma \ \mbox{ on } \, S , \\
 \Psi \; \mbox{ is } 1-\mbox{periodic},
      \end{array}
    \right.
\end{equation}
where $\sigma(y)$ is the adimensionalized charge distribution on the pore surface,
$\beta>0$ is a parameter defined by (\ref{def.beta}) as the square of the ratio of
the pore size and the Debye length,
and $\Phi(\Psi)$ is the bulk density of charges, a nonlinear function defined by
\begin{equation}\label{BPPsi}
     \Phi (\Psi) = -  \sum_{j=1}^N z_j n_j(\Psi)
\end{equation}
where $n_j$ is the concentration of species $j$ given, at equilibrium, as a function
of the potential
\begin{equation}\label{BPPsi2}
n_j (\Psi) =  n_j^0(\infty) e^{ - z_j \Psi  }
\end{equation}
with a constant $n_j^0(\infty) > 0$.
We assume that all valencies $z_j$ are different. If not, we lump together
different ions with the same valence. Of course, for physical reasons,
all valencies $z_j$ are integers. We rank them by increasing order,
$$
z_1 < z_2 < ... < z_N ,
$$
and we denote by $j^+$ and $j^-$ the sets of positive and negative valences.
We also assume that they are both anions and cations, namely positive and negative
valences, i.e., $z_1<0<z_N$. The function $\Phi$ is monotone since it is the
derivative of a convex function
\begin{equation}\label{conv}
\Phi (\Psi) = \Cc^\prime(\Psi) \quad \mbox{ with } \quad
\Cc(\Psi) = \sum_{j=1}^N n_j^0(\infty) e^{- z_j \Psi} .
\end{equation}
The boundary value problem (\ref{BP1}) is equivalent to the following minimization problem:
\begin{equation}
\label{Minpsi}
\inf_{\varphi \in H^1_\#(Y_F)} \left\{ J(\varphi) =
\frac{1}{2} \int_{Y_F} |\nabla \varphi|^2 \, dy + \beta \int_{Y_F} \Cc(\phi) \, dy
+ \int_{S} \sigma \varphi \, dS \right\} ,
\end{equation}
with $H^1_\#(Y_F) = \{ \varphi \in H^1(Y_F) , \; \varphi \; \mbox{is } \, 1-\mbox{periodic} \}$.
The functional $J$ is strictly convex, which gives the uniqueness of the minimizer. Nevertheless, for arbitrary non-negative $\beta , n_j^c$, $J$ may be not coercive on $H^1_\#(Y_F)$ if all $z_j$'s have the same sign. Therefore, we must put a condition on the $z_j$'s so that the minimization problem (\ref{Minpsi}) admits a solution. Following the literature, we impose the bulk electroneutrality condition
\begin{equation}\label{Neutrality}
- \Phi (0) = \sum_{j=1}^N z_j n^0_j(\infty) =0,
\end{equation}
which guarantees that for $\sigma=0$, the unique solution of (\ref{BP1}) is $\Psi =0$.
Under assumption (\ref{Neutrality}) it is easy to see that $J$ is coercive on $H^1_\#(Y_F)$.

\begin{remark}
The bulk electroneutrality condition (\ref{Neutrality}) is not
a restriction. Actually all our results hold under
the much weaker assumption that all valences $z_j$
do not have the same sign. Indeed, if (\ref{Neutrality}) is
not satisfied, we can make a change of variables in the
Poisson-Boltzmann equation (\ref{BP1}), defining a new
potential $\tilde\Psi=\Psi+\Psi^0$ where $\Psi^0$ is a
constant reference potential. Since the function $\Phi(\Psi^0)$
is continuous and admits the following limits at infinity
$$
\lim_{\Psi^0\to\pm\infty} \Phi(\Psi^0) = \pm\infty ,
$$
there exists at least one value $\Psi^0$ such that $\Phi(\Psi^0) =0$.
This change of variables for the potential leaves (\ref{BP1})
invariant if we change the constants $n_j^0 (\infty)$ in new
constants $\tilde n_j^{0} (\infty) =  n_j^0 (\infty) e^{ -  z_j \Psi^0  }$.
These new constants satisfy the bulk electroneutrality condition (\ref{Neutrality}).
\end{remark}

In the sequel we assume that the charge density is a continuous periodic function
\begin{equation}
\label{Bds1}
\sigma (y)\in C_\# (S) .
\end{equation}
The well-posedness of the boundary value problem (\ref{BP1}) is a classical
result.

\begin{lemma}[\cite{L:06}]
\label{lem.pb}
Assume (\ref{Bds1}) and that the electroneutrality condition (\ref{Neutrality}) holds true.
Then problem (\ref{Minpsi}) has a unique solution $\Psi \in H^1_\#(Y_F)$ such that
$$
\sum_{j=1}^N z_j e^{- z_j \Psi}  \; \mbox{and} \; \Psi \sum_{j=1}^N z_j e^{- z_j \Psi }
$$
are absolutely integrable.
\end{lemma}

The issue of the solution boundedness was not correctly addressed in \cite{L:06},
where it was merely proved that $\Psi \in L^p (Y_F)$ for every finite $p$.
By using elementary comparison arguments
we will prove a $L^\infty$-estimate in Proposition \ref{Linft} 
(a similar result is also proved in \cite{EJL}).
To this end we introduce the following auxiliary Neumann problem
\begin{equation}
\label{BU1}
\left\{  \begin{array}{ll}
        \dsp - \Delta U =\frac{1}{| Y_F |} \int_S \sigma \ dS \ \mbox{ in } \ Y_F , &  \\
        \dsp   \nabla U  \cdot \mathbf{n} = - \sigma\ \mbox{ on } \, S ,& \\
 U \; \mbox{ is } 1-\mbox{periodic} , \quad \int_{Y_F}  U (y) \ dy =0.&
      \end{array}
    \right.
\end{equation}
Remark that (\ref{BU1}) admits a solution $U\in H^1_\#(Y_F)$ since the
bulk and surface source terms are in equilibrium. Furthermore, the zero
average condition of the solution gives its uniqueness. Under
condition (\ref{Bds1}) it is known that $U$ is continuous and achieves its
minimum and maximum in $\overline{Y_F}$.

Then our $L^\infty$-bound reads as follows
\begin{proposition}
\label{Linft}
The solution $\Psi$ of (\ref{BP1}) satisfies the following bounds
\begin{gather}
U(y) - U_m - \frac{1}{z_1} \log \max \left( 1 , \,
\frac{\overline{\sigma}}{\beta z_1 n^0_1(\infty)}
- \sum_{j\in j^+} \frac{z_j n^0_j(\infty)}{z_1 n^0_1(\infty)}  \right)
\geq \Psi (y)  \geq  \notag \\
U (y) - U_M  - \frac{1}{z_N} \log \max \left( 1 , \,
\frac{\overline{\sigma}}{\beta z_N n^0_N(\infty)} -
\sum_{j\in j^-} \frac{z_j n^0_j(\infty)}{z_N n^0_N(\infty)}  \right) ,
\label{Linfty}
\end{gather}
where the sets $j^+$ and $j^-$ denote the sets of positive and negative valences, respectively,
and
$$
\overline{\sigma}= \frac{1}{| Y_F |} \int_S \sigma \ dS \ , \quad
U_m =\min_{y\in \overline{Y_F}} U(y) \ \mbox{ and } \quad   U_M =\max_{y\in \overline{Y_F}} U(y) .
$$
\end{proposition}

\begin{proof}
We use the variational formulation for $\Psi- U$, which reads, for any
smooth 1-periodic function $\varphi$,
\begin{gather}
\int_{Y_F} \nabla (\Psi - U) \cdot\nabla \varphi  \ dy +
\beta \int_{Y_F} \Phi (\Psi) \varphi \ dy  + \overline{\sigma} \int_{Y_F} \varphi \ dy =0  .
\label{VVP}
\end{gather}
We take $\varphi (y) = (\Psi (y) -U(y) + C)^-$, where $C$ is a constant to be determined
and, as usual, the function $f^-=\min(f,0)$ is the negative part of $f$. We note that
by Lemma  \ref{lem.pb}, $\varphi \in H^1_\#(Y_F)$ and $\Phi (\Psi) \varphi$ is integrable. It yields
\begin{gather*}
\int_{Y_F} | \nabla \varphi |^2 \, dy + \beta \int_{Y_F}
\left( \Phi (\Psi) - \Phi (U-C) \right) (\Psi - U + C)^- \, dy  \\
+ \int_{Y_F} \left( \beta \Phi (U -C) + \overline{\sigma} \right) \varphi \, dy =0.
\end{gather*}
By monotonicity of $\Phi$ the second term is nonnegative. Hence we should choose $C$ in
such a way that the coefficient in front of $\varphi$ in the third term is nonpositive.
If $C\geq U_M$, we have
\begin{gather*}
\beta \Phi (U -C) +  \overline{\sigma} \leq -\beta z_N n^0_N(\infty) e^{z_N (C - U_M )}
- \beta \sum_{j\in j^-} z_j n^0_j(\infty)  + \overline{\sigma} .
\end{gather*}
Thus, if it happens that
\begin{gather}
-\beta \sum_{j\in j^-} z_j n^0_j(\infty)  + \overline{\sigma} < \beta z_N n^0_N(\infty) ,
\label{condC}
\end{gather}
we indeed take $C=U_M$ and the corresponding term is nonpositive.
If (\ref{condC}) is not true, then our choice is
$$
C= U_M + \frac{1}{z_N} \log \left( \frac{\overline{\sigma}}{\beta z_N n^0_N(\infty)}
- \sum_{j\in j^-} \frac{z_j n^0_j(\infty)}{z_N n^0_N(\infty)} \right) \geq U_M .
$$
This choice of the constant $C$ implies that $\varphi (y) =0$ and it yields
the lower bound in (\ref{Linfty}).

Let us switch to the upper bound.
We now take $\varphi (y) = (\Psi (y) -U(y) -C)^+$, where $C$ is another constant
to be determined. We note that by Lemma  \ref{lem.pb}, $\varphi \in H^1_\#(Y_F)$ and
$\Phi (\Psi) \varphi$ is integrable. It yields
\begin{gather*}
\int_{Y_F} | \nabla \varphi |^2 \ dy + \beta \int_{Y_F}
( \Phi (\Psi) - \Phi (U+C) ) (\Psi - U- C)^+ \ dy  \\
+ \int_{Y_F} ( \beta \Phi (U +C) + \overline{\sigma} ) \varphi \ dy =0.
\end{gather*}
By monotonicity the second term is nonnegative. Again we should choose $C$ in
such a way that the coefficient in front of $\varphi$ in the third term is nonnegative.
If $C+U_m \geq 0$, we have
\begin{gather*}
\beta \Phi (U(y) +C) + \overline{\sigma} \geq - \beta z_1 n^0_1(\infty) e^{-z_1 (C + U_m )}
- \beta \sum_{j\in j^+} z_j n^0_j(\infty) + \overline{\sigma} .
\end{gather*}
Thus, if it happens that
\begin{gather}
\beta \sum_{j\in j^+} z_j n^0_j(\infty) - \overline{\sigma} < - \beta z_1 n^0_1(\infty) ,
\label{condC2}
\end{gather}
we take $C=-U_m$. If (\ref{condC2}) is not true, then our choice is
$$
C=- U_m - \frac{1}{z_1} \log \left(
\frac{\overline{\sigma}}{\beta z_1 n^0_1(\infty)}
- \sum_{j\in j^+} \frac{z_j n^0_j(\infty)}{z_1 n^0_1(\infty)} \right) \geq -U_m .
$$
This choice of the constant $C$ gives the upper bound in (\ref{Linfty}) and the Proposition is proved.
\end{proof}

By classical regularity theory for elliptic partial differential equations,
we easily deduce from Proposition \ref{Linft} that the solution of the
Poisson-Boltzmann equation is as smooth as the data are.

\begin{corollary}
Suppose $S\in C^\infty$ and $\sigma \in C^\infty_\# (S)$.
Then $\Psi \in C^\infty ({\bar Y}_F)$.
\end{corollary}

\section{The limit case of small $\beta$}
\label{small}

In this section we investigate the case of small values of $\beta$ which occurs
for small pore sizes. Equivalently, in the case of very dilute concentrations,
we can scale all concentration coefficients $n_j^0 (\infty)$ by a small parameter
which is multiplied to $\beta$. We shall prove that the solution $\Psi=\Psi_\beta$
of the Poisson-Boltzmann equation (\ref{BP1}) (with a subscript $\beta$ to indicate
that we study the behavior when $\beta \to 0^+$) is uniformly bounded up to an
additive constant which may blow up.

\begin{lemma}
\label{unif}
Let $\Psi_\beta$ be the unique solution of (\ref{BP1}). There exist a constant $C$,
which does not depend on $\beta$, such that
\begin{gather}
\label{Poin}
\| \Psi_\beta - \mathcal{M} (\Psi_\beta) \|_{H^1 (Y_F)} \leq C \|\sigma\|_{L^2(S)} .
\end{gather}
\end{lemma}

\begin{proof}
We recall that the variational formulation (or the virtual work formulation)
corresponding to (\ref{BP1}) is, for any smooth 1-periodic function $\varphi$,
\begin{gather}
\int_{Y_F} \nabla \Psi_\beta\cdot\nabla \varphi \, dy +
\beta \int_{Y_F} \Phi (\Psi_\beta) \varphi \, dy + \int_S \sigma\varphi \, dS =0 ,
\label{VV}
\end{gather}
where the nonlinear function $\Phi$ is defined by (\ref{BPPsi}).
Let $\mathcal{M}$ be the averaging operator defined by
$\mathcal{M}(g) =  \frac{1}{|Y_F |} \int_{Y_F} g (y) \ dy$.
Taking the test function $\dsp \varphi = \Psi_\beta-  \mathcal{M} (\Psi_\beta )$ in (\ref{VV}) we get
\begin{gather*}
\int_{Y_F} | \nabla\Psi_\beta |^2 \, dy +
\beta \int_{Y_F} \left(\Phi (\Psi_\beta) - \Phi (\mathcal{M} (\Psi_\beta) \right)
\left( \Psi_\beta - \mathcal{M} (\Psi_\beta ) \right) \, dy  \notag \\
  + \int_S \sigma(\Psi_\beta -  \mathcal{M} (\Psi_\beta ) )\ dS =0 .
\end{gather*}
By monotonicity of $\Phi$ the second term is nonnegative and Poincar\'e inequality
yields the a priori estimate (\ref{Poin}).

\end{proof}

When $\int_S \sigma \, dS\neq0$, we expect that $\mathcal{M} (\Psi_\beta)$ blows
up as $\beta$ tends to zero. This is already indicated by the $L^\infty$-bounds
from Proposition \ref{Linft}. More convincingly, choosing $\varphi =1$ in
(\ref{VV}) leads to
\begin{equation}\label{Balance}
\beta \int_{Y_F} \Phi (\Psi_\beta)  \ dy  =- \int_S \sigma \ dS,
\end{equation}
which shows that $\mathcal{M} (\Phi (\Psi_\beta))$ blows up, at least.
Remark also that, for $\beta =0$, the corresponding boundary value problem (\ref{BP1}) has no solution. However, in the case $\int_S \sigma \, dS=0$, $\Psi_\beta$ is bounded
as will be proved in Section \ref{ss=0}.

\subsection{Formal asymptotics}
\label{ss.formal}
We first obtain by a formal method of asymptotic expansions the boundary value problem,
corresponding to the limit $\beta \to 0$.  There are 3 possibilities.
  \vskip2pt
  {\bf Case 1: $\int_S \sigma <0$.} In this case it is the most negative valence $z_1$ which matters.
Since $\Psi_\beta -  \mathcal{M} (\Psi_\beta ) $ is bounded, we look for an asymptotic formula
\begin{equation}
\label{eq.ansatz}
\Psi_\beta(y) = a_\beta +  \psi_0 (y) + o(1) ,
\end{equation}
with a constant $a_\beta \to +\infty $ and $\psi_0$ a function independent of $\beta$
and of zero mean in $Y_F$.
We will see that further terms in the asymptotic expansion come in general with
fractional powers of $\beta$.

The equality (\ref{Balance}) becomes
\begin{equation*}
    -\beta \sum_{j=1}^N z_j n^0_j (\infty) \int_{Y_F} e^{-z_j (a_\beta + \psi_0 (x) + o(1))}  \ dy = - \int_S \sigma \ dS .
\end{equation*}
Therefore, we have
\begin{gather}
z_1 n^0_1(\infty) \beta e^{-z_1 a_\beta}\int_{Y_F} e^{z_1 (\psi_0(y) + o(1))}  \ dy \notag \\
+ \beta \sum_{j=2}^N z_j n^0_j(\infty) e^{-z_j a_\beta } \int_{Y_F} e^{-z_j  (\psi_0 (y) + o(1))}  \ dy = \int_S \sigma \ dS ,
\label{Balance0}
\end{gather}
and, since $a_\beta \to +\infty $, at the leading order (\ref{Balance0}) reads
\begin{equation}
\label{Balance1}
z_1 n^0_1(\infty) \beta e^{-z_1 a_\beta} \int_{Y_F} e^{-z_1 (\psi_0(y) + o(1))} \ dy = \int_S \sigma \ dS .
\end{equation}
From (\ref{Balance1}) we deduce
\begin{equation}
\label{alg1}
    a_\beta = \frac{\log\beta}{z_1} + C_0 ,
\end{equation}
where $C_0$ is a constant which may depend on $\beta$ but is bounded as a function of $\beta\to0$. This decomposition allows us to eliminate the singular part $a_\beta$ in the expansion.
We thus get the following nonlinear equation for $\psi_0(y)$
\begin{equation}
\label{BP1ren}
\left\{ \begin{array}{ll}
\dsp  -\Delta \psi_{0} (y) -\int_S \sigma \ dS \frac{e^{-z_1 \psi_0} }{\int_{Y_F} e^{-z_1 \psi_{0} (y)} \ dy} =0 \ \mbox{ in } \ Y_F , &  \\
        \dsp   \nabla \psi_{0} \cdot {\bf n} = -\sigma\ \mbox{ on } \, S ,& \\
\dsp \psi_{0} \; \mbox{ is } 1-\mbox{periodic and } \int_{Y_F} \psi_{0} \ dy=0.&
      \end{array}
    \right.
\end{equation}
In our approximate formula for $\Psi_\beta$ we have neglected terms of order $O (\beta^{1- z_2 /z_1 } )$. In the canonical case of 2 opposite valencies ($N=2$, $z_1=-z_2$), the neglected term is of order $O(\beta ^2)$.

Equation (\ref{BP1ren}) does not contain $\beta$.
Rather than using $\psi_0$, it is more practical to use $\varphi_0 (y)= \psi_0 (y) + C_0$. Then we have
\begin{equation}
\label{Smallbeta1}
    \Psi_\beta (y)  = \frac{\log\beta}{z_1} + \varphi_0 (y) + O(\beta^{1/z^-} ),
\end{equation}
and $\varphi_0$  is the solution to the boundary value problem
\begin{equation}\label{BP0}
\left\{ \begin{array}{ll}
\dsp  -\Delta \varphi_0 (y) - z_1 n^0_1(\infty) e^{-z_1 \varphi_0(y)} =0  \ \mbox{ in } \ Y_F , &  \\
        \dsp   \nabla  \varphi_0 \cdot {\bf n} = -\sigma\ \mbox{ on } \, S ,& \\
 \varphi_0 \; \mbox{ is } 1-\mbox{periodic}.&
      \end{array}
    \right.
\end{equation}
We note that by testing (\ref{BP0}) by a constant and integrating, we get
$$
z_1 n^0_1(\infty) \int_{Y_F} e^{-z_1 \varphi_0(y)} \ dy = \int_S \sigma \ dS.
$$
Consequently, $\varphi_0$ solves (\ref{BP1ren}) except that it is not of mean zero.
We have the following simple result.

\begin{lemma}
\label{L11}
Assume that $\sigma$ is a smooth bounded function such that $\int_S \sigma \ dS <0$.
Then problem (\ref{BP0}) has a unique solution $\varphi_{0} \in H^1_\#(Y_F)$ such that
$$
 e^{-z_1 \varphi_{0}}  \; \mbox{and} \;  e^{-z_1 \varphi_{0}} \varphi_{0}
$$
are absolutely integrable.
\end{lemma}

\begin{proof}
The corresponding functional, to be minimized, is
\begin{gather*}
J_0(\varphi) = \frac{1}{2} \int_{Y_F} | \nabla \varphi |^2 \ dy + n^0_1(\infty)
\int_{Y_F} e^{-z_1 \phi} \  dy +  \int_{S} \sigma \varphi \ dS .
\end{gather*}
It is strictly convex and the condition $\int_S \sigma \ dS <0$ insures the coercivity.
The rest of the proof follows that of Lemma \ref{lem.pb}.
\end{proof}

\vskip2pt
  {\bf Case 2: $\int_S \sigma >0$.} In this case it is the largest positive valence $z_N$
which matters. At the leading order, (\ref{Balance}) reads
\begin{equation}\label{Balance2}
\beta z_N n^0_N(\infty) e^{-z_N a_\beta}\int_{Y_F} e^{-z_N (\psi_0 (y) + o(1)))} \ dy
= \int_S \sigma \ dS .
\end{equation}
For the same asymptotic expansion (\ref{eq.ansatz}),
equation (\ref{Balance2}) allows us to compute the singular behavior $a_\beta$
and we get the following equation for the zero-mean perturbation $\psi_{0}$, $\dsp \int_{Y_F} \psi_{0} \ dy=0$
\begin{equation}
\label{BP1ren2}
\left\{ \begin{array}{ll}
        \dsp  -\Delta \psi_{0} -\int_S \sigma \ dS \frac{e^{-z_N \psi_0} }{\int_{Y_F} e^{-z_N \psi_{0} (x)} \ dy} =0 \ \mbox{ in } \ Y_F , &  \\
        \dsp   \nabla \psi_{0} \cdot {\bf n} = -\sigma\ \mbox{ on } \, S ,& \\
 \psi_{0} \; \mbox{ is } 1-\mbox{periodic and } \int_{Y_F} \psi_{0} \ dy=0.&
      \end{array}
    \right.
\end{equation}
By the same reasoning as in the first case, we deduce
\begin{equation}
\label{Smallbeta}
\Psi_\beta (y) = \frac{\log\beta}{z_N} + \xi_{0} (y) + O(\beta^{1/z_N} ),
\end{equation}
where $\xi_0$ is the solution of
\begin{equation}\label{BP00}
\left\{ \begin{array}{ll}
\dsp  -\Delta \xi_0 (y) - z_N n^0_N(\infty) e^{-z_N \xi_0 (y) } =0  \ \mbox{ in } \ Y_F , &  \\
\dsp   \nabla  \xi_0 \cdot {\bf n} = -\sigma\ \mbox{ on } \, S ,& \\
 \xi_0 \; \mbox{ is } 1-\mbox{periodic}.&
\end{array}  \right.
\end{equation}
By testing (\ref{BP00}) with a constant and integrating, we get
$$
z_N n^0_N(\infty) \int_{Y_F} e^{-z_N \xi_0 (y) } \ dy = \int_S \sigma \ dS.
$$
Consequently, $\xi_0$ solves (\ref{BP1ren2}) except that it is not of zero average.
We have the following simple result.

\begin{lemma}
\label{L12}
Assume that $\sigma$ is a smooth bounded function such that $\int_S \sigma \ dS >0$.
Then problem (\ref{BP00}) has a unique solution $\varphi_{0} \in H^1_\#(Y_F)$ such that
$$
 e^{- z_N \xi_{0}}  \; \mbox{and} \;  e^{- z_N \xi_{0}} \xi_{0}
$$
are absolutely integrable.
\end{lemma}

\vskip2pt
  {\bf Case 3: $\int_S \sigma =0$.} In this case the problem corresponding to $\beta=0$ has a solution and the analysis is much simpler.

The following limit problem
\begin{equation}\label{BP1ren3}
\left\{ \begin{array}{ll}
\dsp  -\Delta \Psi_{0} (y) =0 \ \mbox{ in } \ Y_F , &  \\
\dsp   \nabla \Psi_{0} \cdot {\bf n} = -\sigma\ \mbox{ on } \, S ,& \\
\dsp \Psi_{0} \; \mbox{ is } 1-\mbox{periodic and } \int_{Y_F} \Phi (\Psi_{0} ) \ dy=0 ,&
\end{array} \right.
\end{equation}
has a unique solution $\Psi_{0}$ since the function $\Phi$ is monotone.
Then we have
\begin{equation}\label{Smallbeta3}
    \Psi_{\beta} (y) = \Psi_{0} (y) + O(\beta).
\end{equation}

\subsection{Rigorous perturbation results when $\int_S \sigma \ dS \neq0$}

We focus on the case $\int_S \sigma \ dS <0$: the opposite one, $\int_S \sigma \ dS >0$,
is completely analogous. Motivated by the discussion leading to (\ref{Smallbeta1}), we
look for $\Psi_\beta$ in the form
\begin{equation}
\label{Smallbetaneg}
\Psi_\beta (y) = \frac{\log\beta}{z_1} + \varphi_\beta (y),
\end{equation}
where $\varphi_\beta$ is the solution of
\begin{equation}
\label{BPbeta-}
\left\{ \begin{array}{ll}
        \dsp  -\Delta \varphi_\beta (y) - z_1 n^0_1(\infty) e^{-z_1 \varphi_\beta (y) } + {\tilde \Phi } (\varphi_\beta ) =0  \ \mbox{ in } \ Y_F , &  \\
        \dsp   \nabla  \varphi_\beta \cdot {\bf n} = -\sigma\ \mbox{ on } \, S ,& \\
 \varphi_\beta \; \mbox{ is } 1-\mbox{periodic,}&
\end{array} \right.
\end{equation}
with
\begin{equation}
\label{rescPhi}
{\tilde \Phi} (g) = - \sum_{j=2}^N z_j n^0_j(\infty) \beta^{1-z_j/z_1} e^{-z_j g} .
\end{equation}
We start with a uniform $H^1$-estimate for $\varphi_\beta$.

\begin{lemma}
\label{L13}
Let $\sigma$ be a smooth bounded function such that $\int_S \sigma \ dS <0$.
Then, for small enough $\beta$, the solution $\varphi_\beta$ of (\ref{BPbeta-})
satisfies the estimate
\begin{equation}\label{Apriori1}
\| \varphi_\beta \|_{H^1 (Y_F)} \leq C,
\end{equation}
where $C$ is independent of $\beta$.
\end{lemma}

\begin{proof}
The variational formulation of problem (\ref{BPbeta-}) reads, for any
smooth 1-periodic function $\varphi$,
\begin{gather}
\label{VVR}
\int_{Y_F} \nabla \varphi_\beta\cdot\nabla \varphi \, dy +
\int_{Y_F} \left( {\tilde \Phi} (\varphi_\beta) -
z_1 n^0_1(\infty) e^{-z_1 \varphi_\beta} \right) \varphi \, dy +
\int_S \sigma\varphi \ dS =0  .
\end{gather}
In (\ref{VVR}) we take $\varphi=\varphi_\beta=\varphi_\beta^++\varphi_\beta^-$ and we get
\begin{equation}\label{trudno4}
\begin{array}{c}
\displaystyle
\int_{Y_F} |\nabla \varphi_\beta^+ |^2 \, dy +
(z_1)^2 n^0_1(\infty) \int_{Y_F} |\varphi_\beta^+|^2 \, dy +
\beta \Phi (\frac{\log\beta}{z_1} ) \mathcal{M} (\varphi_\beta^+) \\[4mm]
\displaystyle
+ \int_S \sigma\varphi_\beta^+ \, dS +  \int_{Y_F} |\nabla \varphi_\beta^- |^2 \, dy +
\beta \Phi (\frac{\log\beta}{z_1} ) \mathcal{M} (\varphi_\beta^-) \\[4mm]
\displaystyle
+ \int_S \sigma(\varphi_\beta^- - \mathcal{M} (\varphi_\beta^-) ) \, dS
+ \left(\int_S \sigma\, dS \right) \mathcal{M} (\varphi_\beta^-)   \leq 0.
\end{array}
\end{equation}
Indeed,
\begin{equation}\label{trudno0}
\int_{Y_F} |\nabla \varphi_\beta |^2 \, dy=
\int_{Y_F} |\nabla \varphi_\beta^+ |^2 \, dy+
\int_{Y_F} |\nabla \varphi_\beta^- |^2 \, dy,
\end{equation}
and
\begin{equation}\label{trudno00}
\int_S \sigma\varphi_\beta \, dS=
\int_S \sigma\varphi_\beta^+ \, dS+
\int_S \sigma(\varphi_\beta^- - \mathcal{M} (\varphi_\beta^-) ) \, dS +\left(\int_S \sigma\, dS \right) \mathcal{M} (\varphi_\beta^-).
\end{equation}
Furthermore, both functions ${\tilde \Phi}(g)$ and $g\to -z_1 e^{-z_1g}$ are
monotone and
$$
{\tilde \Phi} (0) - z_1 n^0_1(\infty) = \beta \Phi (\frac{\log\beta}{z_1} ) .
$$
Thus, we deduce
\begin{equation}\label{trudno1}
\left( {\tilde \Phi} (\varphi_\beta) - z_1 n^0_1(\infty) e^{-z_1 \varphi_\beta}
- {\tilde \Phi} (0) + z_1 n^0_1(\infty)\right) \varphi_\beta \geq 0 .
\end{equation}
However, we use a further argument of {\bf strict monotonicity} for $-z_1 e^{-z_1g}$,
namely
\begin{gather}
\left( -z_1 n_1^0(\infty) e^{-z_1\varphi_\beta} + z_1 n^0_1(\infty)\right) \varphi_\beta^+
= \left( -z_1 n_1^0(\infty) e^{-z_1\varphi_\beta^+} + z_1 n^0_1(\infty)\right) \varphi_\beta^+ \notag \\
\geq (z_1)^2 n_1^0(\infty)\varphi_\beta^+ ,
\label{trudno2}
\end{gather}
because $\left(-z_1 e^{-z_1g}\right)^\prime=(z_1)^2e^{-z_1g}\geq (z_1)^2$ for $g\geq0$.
Equalities (\ref{trudno0}), (\ref{trudno00}), together with the lower bounds (\ref{trudno1}),
(\ref{trudno2}), applied to the variational formulation (\ref{VVR}), yield the
desired inequality (\ref{trudno4}).
We recall that
$$
\lim_{\beta\to0^+} \beta \Phi (\frac{\log\beta}{z_1} ) = -z_1 n^0_1(\infty) > 0,
$$
so that, for sufficiently small $\beta>0$, $\beta \Phi (\frac{\log\beta}{z_1})$
is a positive bounded constant. Further, the product $(\int_S \sigma\ dS) \mathcal{M} (\varphi_\beta^-)$ is nonnegative. Therefore it suffices to apply Poincar\'e inequality and (\ref{Apriori1}) follows.
\end{proof}

Next we need a uniform $L^\infty$-bound for $\varphi_\beta$, as $\beta$ goes to 0.
(Recall that the $L^\infty$-bounds of Proposition \ref{Linft} are not uniform with
respect to $\beta$.)

\begin{proposition}
\label{Linftbetasmall}
For sufficiently small $\beta >0$, we have the bounds
\begin{gather}
U(y) - U_m
- \frac{1}{z_1} \log \max \{ 1, \, \frac{\overline{\sigma}}{z_1 n^0_1(\infty)} \}
\geq \varphi_\beta (y)  \geq  \notag \\
U (y) - U_M - \frac{1}{z_1} \log \min \{ 1 , \,
    \frac{\overline{\sigma}}{z_1 n^0_1(\infty)} \},
\label{Linftybeta1}
\end{gather}
where $U$ is the solution of the Neumann problem (\ref{BU1}).
\end{proposition}

\begin{proof}
We start with the variational formulation for $\varphi_\beta- U$ which reads,
for any smooth 1-periodic function $\varphi$,
\begin{gather}
\int_{Y_F} \nabla (\varphi_\beta - U) \cdot\nabla \varphi \, dy -
z_1 n^0_1(\infty) \int_{Y_F} e^{-z_1 \varphi_\beta} \varphi \, dy +
\int_{Y_F} {\tilde \Phi} (\varphi_\beta) \varphi \, dy \notag \\
+ \overline{\sigma} \int_{Y_F} \varphi \, dy =0 \, . \label{VVB1}
\end{gather}
We take $\varphi (y) = (\varphi_\beta (y) -U(y) + C)^-$, where $C$ is a constant
to be determined. By virtue of Lemma \ref{lem.pb}, $\varphi \in H^1_\#(Y_F)$ and
$\Phi (\varphi_\beta) \varphi$ are integrable. Since the function
$g\to -z_j e^{-z_jg}$ and $g\to {\tilde \Phi}(g)$ are monotone, we deduce from (\ref{VVB1})
\begin{gather*}
\int_{Y_F} | \nabla \varphi |^2 \, dy + \int_{Y_F} \left( -z_1 n^0_1(\infty) e^{-z_1 (U-C) }
+ {\tilde \Phi } (U-C) + \overline{\sigma} \right) \varphi  \, dy \leq 0.
\end{gather*}
Hence we want to choose $C$ such that the expression in front of $\varphi$
in the second integral is nonpositive. We have
\begin{gather}
- z_1 n^0_1(\infty) e^{-z_1 (U(y)-C) } +
{\tilde \Phi} (U(y) -C) + \overline{\sigma}  \leq \notag\\
- z_1 n^0_1(\infty) e^{-z_1 (U_M-C)}
-\sum_{z_1 < z_j<0} \beta^{1-z_j / z_1}  z_j n^0_j(\infty) e^{-z_j (U_M-C)} + \overline{\sigma} .
\label{Linftybeta5}
\end{gather}
Now if
$$
\overline{\sigma} < z_1 n^0_1(\infty) < 0 ,
$$
we take $C=U_M$ and the left hand side of (\ref{Linftybeta5}) is nonpositive
for $\beta$ sufficiently small because the sum in (\ref{Linftybeta5}) is small.
If not, then our choice is
$$
C > U_M + \frac{1}{z_1} \log \big( \frac{\overline{\sigma}}{z_1 n^0_1(\infty)}\big) > U_M .
$$
This choice of the constant $C$ implies that $\varphi (y) = 0$, for small
enough $\beta$, and yields
the lower bound in (\ref{Linftybeta1}).

For the upper bound we take $\varphi (y) = (\varphi_\beta (y) -U(y) -C)^+$,
where $C$ is a constant to be determined. It yields
\begin{gather*}
\int_{Y_F} | \nabla \varphi |^2 \, dy + \int_{Y_F} \left( -z_1 n^0_1(\infty) e^{-z_1 (U+C)}  + {\tilde \Phi } (U+C) + \overline{\sigma} \right) \varphi \, dy \leq 0.
\end{gather*}
Hence we should choose $C$ such that the expression in front of $\varphi$
in the second integral is nonnegative. We have
\begin{gather}
\label{Linftybeta6}
-z_1 n^0_1(\infty) e^{-z_1 (U(y)+C)} + {\tilde \Phi} (U+C) + \overline{\sigma}  \geq  \\
-z_1 n^0_1(\infty) e^{-z_1 (U_m+C) }  - \sum_{j\in j^+} \beta^{1-z_j / z_1 } z_j n^0_j(\infty)
e^{-z_j (U_m+C) } + \overline{\sigma} . \notag
\end{gather}
Now if
$$
z_1 n^0_1(\infty) < \overline{\sigma} < 0,
$$
we choose $C=-U_m$ and, for sufficiently small $\beta$, the right hand side of
(\ref{Linftybeta6}) is positive because the sum over $j\in j^+$ is small
and the expression in front of $\varphi$
in the second integral is nonnegative. Otherwise, we choose
$$
C > - U_m - \frac{1}{z_1} \log\left(\frac{\overline{\sigma}}{z_1 n^0_1(\infty)}\right) > - U_m
$$
and, again, for sufficiently small $\beta$,  the right hand side of
(\ref{Linftybeta6}) is positive, which implies the upper bound in (\ref{Linftybeta1}).
\end{proof}

As an immediate consequence of Proposition \ref{Linftbetasmall},
taking the limit as $\beta$ goes to 0, we obtain the following corollary.

\begin{corollary}
\label{Linftbetasmall0}
Let $\varphi_0$ be the solution of (\ref{BP0}). It satisfies the $L^\infty$-estimate
\begin{gather}
U(y) - U_m
- \frac{1}{z_1} \log \max \{ 1, \, \frac{\overline{\sigma}}{z_1 n^0_1(\infty)} \}
\geq \varphi_0 (y)  \geq  \notag \\
U (y) - U_M - \frac{1}{z_1} \log \min \{ 1 , \,
    \frac{\overline{\sigma}}{z_1 n^0_1(\infty)} \} .
\label{Linftybeta10}
\end{gather}
\end{corollary}

\begin{theorem}
\label{th1}
We have
\begin{equation}\label{Errest}
\| \varphi_\beta - \varphi_0 \|_{C^k ({\bar Y}_F) )} \leq C \beta^{1- z_2/z_1},
\end{equation}
for every positive integer $k$.
Furthermore, let $\varphi_1$ be the solution for
\begin{equation}\label{BP01}
\left\{ \begin{array}{ll}
\dsp  -\Delta \varphi_1 + (z_1)^2 n^0_1(\infty) e^{-z_1 \varphi_0} \varphi_1 =
z_{2} n^0_{2}(\infty) e^{-z_2 \varphi_0}  \ \mbox{ in } \ Y_F , &  \\
\dsp \nabla  \varphi_1 \cdot {\bf n} =0 \ \mbox{ on } \, S ,& \\
\varphi_1 \; \mbox{ is } 1-\mbox{periodic}.&
\end{array} \right.
\end{equation}
Then, for every positive integer $k$, we have
\begin{equation}\label{Errest1}
\| \varphi_\beta - \varphi_0 - \beta^{1- z_2/z_1} \varphi_1 \|_{C^k ({\bar Y}_F) )}
\leq C \beta^q,
\end{equation}
where $0<q=\min\big( 1- z_3/z_1 \, , \, 2(1- z_2/z_1) \big)$.
\end{theorem}

\begin{proof}
First we observe that $\varphi_\beta - \varphi_0$ satisfies the variational equation
\begin{gather}
  \int_{Y_F} \nabla (\varphi_\beta -\varphi_0 ) \cdot\nabla \varphi \, dy
- z_1 n^0_1(\infty) \int_{Y_F} ( e^{-z_1 \varphi_\beta} - e^{-z_1 \varphi_0} ) \varphi \, dy = \notag \\
  - \int_{Y_F} {\tilde \Phi} (\varphi_\beta) \varphi \, dy
    \, \mbox{ for all smooth 1-periodic } \; \varphi .
\label{VVB2}
\end{gather}
Now we take $\varphi = \varphi_\beta -\varphi_0$ as test function, use the strict
monotonicity of the function $g\to -z_j e^{-z_j g}$, the $L^\infty$-bounds (\ref{Linftybeta1})
and  (\ref{Linftybeta10}) to conclude that
\begin{equation}\label{Estzero}
\| \varphi_\beta -\varphi_0 \|_{H^1(Y_F)} \leq C \beta^{1- z_2/z_1} .
\end{equation}
Next we write the equation for $\varphi_\beta -\varphi_0$ as
\begin{equation}
\label{BPbetadif}
\left\{ \begin{array}{ll}
\dsp  -\Delta (\varphi_\beta -\varphi_0) + (\varphi_\beta -\varphi_0) =
(\varphi_\beta -\varphi_0) & \\
+ z_1 n^0_1(\infty)( e^{-z_1 \varphi_\beta} - e^{-z_1 \varphi_0} )
- {\tilde \Phi } (\varphi_\beta )  \ \mbox{ in } \ Y_F , &  \\
\dsp \nabla (\varphi_\beta -\varphi_0) \cdot {\bf n} = 0 \ \mbox{ on } \, S ,& \\
(\varphi_\beta -\varphi_0) \; \mbox{ is } 1-\mbox{periodic}.&
\end{array} \right.
\end{equation}
Using the estimate (\ref{Estzero}), we get the $H^2 -$error estimate of the same order. After bootstrapping we obtain the required error estimate (\ref{Errest}).

Eventually, we write the equation for $\varphi_\beta - \varphi_0 - \beta^{1- z_2/z_1} \varphi_1$
and repeating the above procedure yields (\ref{Errest1}).
\end{proof}

\begin{remark}
In the frequently considered case of two ions of opposite unit charge ($N=2$, $-z_1=z_2=1$),
normalizing the coefficients $n^0_1(\infty)=n^0_2(\infty)=1$, we have
$$
\Phi (g) = 2 \sinh g \ , \quad  \varphi_\beta = \varphi_0 +\beta^2 \varphi_1 +\beta^4 \varphi_2 +\dots
$$
and, in the case $\int_S \sigma \ d S <0$, the equations for the functions $\varphi_j$ read
\begin{gather*}
    -\Delta \varphi_0 +  e^{\varphi_0} =0,\\
     -\Delta \varphi_1 +  e^{\varphi_0} \varphi_1 =  e^{-\varphi_0},\\
      -\Delta \varphi_2 +  e^{\varphi_0} \varphi_2 = - e^{\varphi_0} \varphi_1^2 - e^{-\varphi_0} \varphi_1
\end{gather*}
and we have
\begin{equation}\label{exsh}
\| \varphi_\beta - \varphi_0 - \beta^2 \varphi_1 - \beta^4 \varphi_2 \|_{C^k ({\bar Y}_F) )} \leq C \beta^6,
\end{equation}
for every positive integer $k$.
\end{remark}

The case $\int_S \sigma \ dS >0$ is analogous and it is enough to repeat the above strategy
with $z_1$ replaced by $z_N$.

\subsection{Rigorous perturbation results in the case $\int_S \sigma \ dS =0$}
\label{ss=0}

Here the proofs are much simpler than in the previous subsection. We just state the results. Again, the starting point are the uniform $H^1$-estimate for $\Psi_\beta$.

\begin{lemma}
\label{L14}
Let $\sigma$ be a smooth function such that $\int_S \sigma \ dS =0$.
Then the solution $\Psi_\beta$ of problem (\ref{BP1}) satisfies the uniform estimates
\begin{gather}
U(y) - U_m - \frac{1}{z_1} \log \max \left( 1 , \,
- \sum_{j\in j^+} \frac{z_j n^0_j(\infty)}{z_1 n^0_1(\infty)}  \right)
\geq \Psi (y)  \geq  \notag \\
U (y) - U_M  - \frac{1}{z_N} \log \max \left( 1 , \, -
\sum_{j\in j^-} \frac{z_j n^0_j(\infty)}{z_N n^0_N(\infty)}  \right) ,
\label{Linftybeta14}
\end{gather}
and
\begin{equation}\label{Apriori14}
\| \Psi_\beta \|_{H^1 (Y_F)} \leq C,
\end{equation}
where $C$ is independent of $\beta$.
\end{lemma}

\begin{proof}
The $L^\infty$-bound (\ref{Linftybeta14}) is a direct consequence of
Proposition \ref{Linft}. Note that (\ref{Linftybeta14}) is uniform with respect
to $\beta$. To obtain (\ref{Apriori14}) we take the test function
$\varphi=\Psi_\beta$ in the variational formulation (\ref{VV})
$$
\int_{Y_F} | \nabla \Psi_\beta |^2 \, dy + \beta \int_{Y_F} \Phi(\Psi_\beta) \Psi_\beta \, dy
+ \int_S \sigma \Psi_\beta \, dS =0 .
$$
Since $\Phi$ is monotone and satisfies $\Phi(0)=0$, we have $\Phi(\Psi_\beta) \Psi_\beta\geq0$,
while the assumption $\int_S \sigma \ dS =0$ implies that
$$
\int_S \sigma \Psi_\beta \, dS =
\int_S \sigma \left( \Psi_\beta - {\cal M}(\Psi_\beta) \right) dS
\leq C \| \sigma \|_{L^2(S)} \| \nabla \Psi_\beta \|_{L^2(Y_F)}
$$
by virtue of Poincar\'e-Wirtinger inequality. We thus deduce
$$
\| \nabla \Psi_\beta \|_{L^2(Y_F)} \leq C \| \sigma \|_{L^2(S)} ,
$$
which, together with (\ref{Linftybeta14}), implies (\ref{Apriori14}).
\end{proof}

In subsection \ref{ss.formal} we already introduced the limit problem (\ref{BP1ren3})
for $\Psi_0=\lim_{\beta\to0}\Psi_\beta$. We can also define a corrector $\Psi_1$
as the unique solution of
\begin{equation}\label{BP1ren4}
\left\{ \begin{array}{ll}
\dsp  -\Delta \Psi_{1} (y) = - \Phi(\Psi_0) \ \mbox{ in } \ Y_F , &  \\
\dsp   \nabla \Psi_{1} \cdot {\bf n} = 0 \ \mbox{ on } \, S ,& \\
\dsp \Psi_{1} \; \mbox{ is } 1-\mbox{periodic and }
\int_{Y_F} \Phi^\prime(\Psi_{0})\Psi_{1} \, dy=0 .&
\end{array} \right.
\end{equation}
There exists a solution of (\ref{BP1ren4}) because $\int_{Y_F} \Phi(\Psi_{0}) \, dy=0$
as required by the definition of (\ref{BP1ren3}).
As an obvious consequence of Lemma \ref{L14} we get the following error estimate.
\begin{theorem}\label{th2}
Let $\Psi_0$ be the solution of (\ref{BP1ren3}) and $\Psi_1$ that of (\ref{BP1ren4}).
Then we have
\begin{equation}\label{Errest14}
\| \Psi_\beta - \Psi_0 \|_{C^k ({\bar Y}_F) )} \leq C \beta, \quad
\| \Psi_\beta - \Psi_0 -\beta\Psi_1 \|_{C^k ({\bar Y}_F) )} \leq C \beta^2,
\end{equation}
for every positive integer $k$.
\end{theorem}

The proof of Theorem \ref{th2} follows the lines of the proof of Theorem \ref{th1}.

\section{Large $\beta$ limit}
\label{Lb}
We now investigate the asymptotic behavior of $\Psi_\beta$ when the $\beta$ parameter
goes to $+\infty$. In view of its definition (\ref{def.beta}), a large value of $\beta$
corresponds either to a large pore size $L$ or to a small Debye length $\lambda_D$,
but also to a large common value of the concentrations $n_j^0(\infty)$. A similar asymptotic
analysis has been performed in \cite{BCEJ:97}, \cite{PJ:97} in one space dimension. 
In higher space dimension our main tool to obtain the behavior near the solid boundaries 
is the multidimensional boundary layer technique introduced by Vishik and Lyusternik \cite{ViLy}.

\subsection{Formal asymptotics}
\label{Form}
In the Poisson-Boltzmann system (\ref{BP1}) the parameter $\beta$ appears in the partial
differential equation but not in the Neumann boundary condition. This indicates the presence
of boundary layers in the asymptotic analysis, the thickness of which shall be of the
order of $O(1/\beta)$. The usual technique to handle this situation is that of matched
asymptotic expansion. We first consider an outer expansion of the solution $\Psi_\beta$
in $Y_F$, away from the boundary $S$. In a second step we shall construct an inner expansion
of $\Psi_\beta$ in the vicinity of $S$, which is equivalently a boundary layer.

We begin with the outer expansion for $\Psi_\beta$ which reads
$$
\Psi_\beta = \Psi_\infty + \frac{1}{\beta} \Psi_{1, \infty} + \frac{1}{\beta^2} \Psi_{2, \infty} + \dots .
$$
After plugging this ansatz in the Poisson-Boltzmann equation (\ref{BP1}) we get
\begin{gather*}
- \frac{1}{\beta}\Delta \Psi_\beta + \Phi ( \Psi_\beta  ) =
\Phi (\Psi_\infty) + \frac{1}{\beta} \Big( \Phi ' (\Psi_\infty) \Psi_{1, \infty} - \Delta  \Psi_\infty \Big) + \dots =0  \ \mbox{ in } \ Y_F
\end{gather*}
which implies, at the zero order, that $\Phi (\Psi_\infty)=0$.
The electroneutrality condition (\ref{Neutrality}) tells us that $0$ is the
unique root of the monotone function $\Phi$. Therefore we deduce $\Psi_\infty =0$
in $Y_F$. In other words we have
\begin{equation}
\label{Outerexp}
\Psi_\beta (y) = O (\frac{1}{\beta} ) \quad \mbox{in $Y_F$, away from the boundary $S$.}
\end{equation}
In fact we will check rigorously in the next Subsection that this order of magnitude
holds for the $L^1-$norm of $\Psi_\beta$.

We now turn to the inner expansion of $\Psi_\beta$, i.e., its behavior close to $S$,
which is more complicated. We study it locally near a point $y_0\in S$, using the
geometrical setting introduced in Subsection \ref{Geo}. We consider a tubular
neighborhood $Y_F^\mu$ of $S$ with $\mu$ small but much bigger than $\beta^{-1/2}$.
Locally, in a neighborhood $\mathcal{N}(y_0)$ of $y_0$, we make the change of variables
$y \to q=(q',q_d)$, as defined in Subsection \ref{Geo}, which satisfies
$| \nabla_y q_d |=1$ in $\mathcal{N}(y_0)$ and ${\bf n} \cdot \nabla_y q_d=1$ on
$\mathcal{N}(y_0)\cap S$.
The Jacobian $J$ (corresponding to the volume differential change $dy = J dq$) is defined by
\begin{gather}
\label{jacobian}
J = \det \left( \frac{\p y_k}{\p q_j} \right)_{1\leq j,k\leq d} ,
\end{gather}
and the metric matrix (corresponding to the transformation $y\to q$)
\begin{gather}
\label{metric}
K= \left( \sum_{j=1}^d \frac{\p q_k}{\p y_j} \frac{\p q_r}{\p y_j} \right)_{1\leq k,r\leq d} ,
\end{gather}
which satisfies
$$
K_{d,d}=1 , \quad K_{k,d}=0 \quad \mbox{ for } 1\leq k \leq d-1 .
$$
Notice that the coordinates $\mathbf{q}=\mathbf{q}(y)$ are introduced in such a way
that the level sets $\{ q_d=\mbox{const} \}$ and the normal lines $\{ q'=(C_1, \dots ,C_{d-1} \}$
are orthogonal (that is the corresponding tangential hyperplanes and lines are orthogonal).
Since $\nabla_y q_d (y)$ gives the direction of the normal line and
$\nabla_y q_k, \ k=1,2, \dots ,d-1,$ form a basis in the tangential hyperplane,
we have $K_{k,d}=\nabla_y q_k \cdot \nabla_y \mbox{dist} (y,S) =0$ for $k\neq d$.

Differential operators in new coordinates transform as follows:
\begin{gather}
\frac{\p }{\p y_j} = \sum_{k=1}^{d}\, \frac{\p q_k}{\p y_j} \frac{\p}{\ q_k}\,
 , \quad j= 1, \dots , d; \notag \\
\sum_{k=1}^d \frac{\p^2 }{\p y_k^2} = \frac1J \div_q \left( J \, K \, \nabla_q \right)
\quad \mbox{ in } Y_F^\mu\cap \mathcal{N}(y_0) ; \label{laplace}  \\
{\bf n} \cdot \nabla_y = - \frac{\p }{\p q_d} \quad \mbox{ on } S.  \label{normalderivative}
\end{gather}
Applying this change of variables to the variational formulation (\ref{VV})
of the Poisson-Boltzmann system (\ref{BP1}) yields the following equation
in the new coordinates
\begin{gather}
- \div_q  \left( J \, K \, \nabla_q \Psi_\beta\right) + \beta \, J\, \Phi(\Psi_\beta) = 0 .
\label{eq.new}
\end{gather}
Dividing (\ref{eq.new}) by $J$ yields that
the partial differential equation in $Y_F^\mu \cap \mathcal{N}(y_0)$ and
the boundary condition on $S\cap \mathcal{N}(y_0)$ of (\ref{BP1})
transform into
\begin{gather}
-\frac{\p^2 \Psi_\beta}{\p q_d^2} + \beta \Phi(\Psi_\beta)
+ \mbox{ lower order derivatives in } q_d \notag \\
+ \mbox{ second order differential operator in } q' \, = \, 0
\mbox{ in } \ Y_F^\mu \cap \mathcal{N}(y_0) , \label{BP1flat1} \\
\frac{\p \Psi_\beta}{\p q_d} = \sigma  \ \mbox{ on } \, S\cap \mathcal{N}(y_0) .\label{BP1flat2}
\end{gather}
As usual in the method of matched asymptotic expansions, problem (\ref{BP1flat1})-(\ref{BP1flat2}) serves to construct the inner expansion.
Since we expect the thickness of the boundary layer to be of order $O(1/\beta)$,
we search for the inner expansion of the form
\begin{gather}
\Psi_\beta(q',q_d) = \beta^{-1/2} \sum\limits_{j=0}^\infty \beta^{-j/2}
\Psi_j(q',\beta^{1/2} q_d) .\label{expan}
\end{gather}
Expanding the zero order term $\Phi(\Psi_\beta)$ in Taylor series and taking
into account the bulk electroneutrality condition (\ref{Neutrality}),
$\Phi(0)=0$, we obtain
\begin{gather}
\Phi(\Psi_\beta) = - \sum_{k=1}^N z_k n^0_k(\infty) e^{-z_k\Psi_\beta} =
\Phi^\prime(0)\Psi_\beta
+ \frac{1}{2} \Phi^{\prime\prime}(0)(\Psi_\beta)^2+\dots \notag \\
= \beta^{-1/2} \sum_{k=1}^N z^2_k n^0_k(\infty) \Psi_0
+ \beta^{-1} \sum_{k=1}^N z^2_k n^0_k(\infty)\Big(\Psi_1-\frac{1}{2}z_k(\Psi_0)^2
\Big)+\dots \label{zero_o}
\end{gather}
Introducing $\xi=\beta^{1/2} q_d $, substituting  (\ref{expan}) and (\ref{zero_o}) in
(\ref{BP1flat1})-(\ref{BP1flat2}) and collecting power-like terms in the resulting equations,
after straightforward rearrangements we arrive at the following
problem for the main term of the expansion:
\begin{gather}
\frac{d^2}{d\xi^2}\Psi_0(q',\xi) -
\left( \sum_{k=1}^N z_k^2 n^0_k(\infty) \right) \Psi_0(q',\xi) = 0 \; \mbox{ for } \;  \xi>0 ; \label{Inner} \\
\frac{d}{d\xi}\Psi_0 (q',0)=\sigma(q') \; \mbox{ for } \;  \xi=0. \label{Inner2}
\end{gather}
Problem (\ref{Inner})-(\ref{Inner2}) is a second-order ordinary differential equation
on the positive half-line. After matching with the outer solution
$\displaystyle \Psi_\beta=O (\frac{1}{\beta})$, we impose additionally that
$\displaystyle \Psi_0(q',+\infty)=0$. The exact solution of (\ref{Inner})-(\ref{Inner2})
is thus
\begin{gather}
\Psi_0(q',\xi) = \frac{-\sigma( q')}{\sqrt{\Phi^\prime(0)}}
\exp \{ -\xi \sqrt{\Phi^\prime(0)} \}  , \label{Leading1}
\end{gather}
with $\Phi^\prime(0) = \sum_{k=1}^N z_k^2 n^0_k(\infty)$.
Back to the original variables, we get
\begin{gather}
\Psi_\beta (y) = \frac{-\sigma(y)}{\sqrt{\beta \Phi^\prime(0)} }
\exp \{ - d(y) \sqrt{\beta \Phi^\prime(0)} \} + O (\frac{1}{\beta} ) , \label{Leading2}
\end{gather}
where $d(y)$ is the distance between $y$ and $S$.
This asymptotic expansion (\ref{Leading2}) will be justified rigorously in the next subsection.

\subsection{Rigorous error estimate}
\label{REE}

We start with two useful simple inequalities for the nonlinearity $\Phi$.

\begin{lemma}\label{Ineq1}
Let $\Phi$ be given by (\ref{BPPsi}), i.e.,
$\Phi(x)=- \sum_{k=1}^N z_k n^0_k(\infty) e^{-z_k x}$.
There exist positive constants $H,C_0,C_k$ such that, $\forall\, x\in \mathbb{R}$,
\begin{gather}
\Phi (x) \, \mbox{{\rm sign}} (x) \geq H^2 \, |x| , \label{Ineqsign} \\
\Phi' (x) \geq  C_0 + C_k |x|^{k-2} ,\quad k\geq 2 . \label{Ineqderiv}
\end{gather}
\end{lemma}

\begin{proof}
To prove (\ref{Ineqderiv}) we note that
$$
\Phi' (x) \geq (z_1)^2 n^0_1(\infty) e^{-z_1 x} + (z_N)^2 n^0_N(\infty) e^{-z_N x}
$$
with $z_1<0<z_N$, which implies the desired result. Then (\ref{Ineqsign}) follows
from a Taylor expansion of $\Phi(x)$ at 0 and the bulk electroneutrality condition
(\ref{Neutrality}), $\Phi(0)=0$.
\end{proof}

We now prove a priori estimates which improve that of Lemma \ref{unif}.

\begin{lemma}
\label{thm.estim}
Let $\Psi_\beta$ be the unique solution of (\ref{BP1}).
There exists a positive constant $C$ such that, $\forall\,\beta\geq1$,
\begin{gather}
\| \Psi_\beta \|_{L^1 (Y_F)} \leq \frac{C}{\beta} , \label{L10} \\
\| \Psi_\beta \|_{L^k (Y_F)} \leq C \beta^{-3/(2k)} , \quad k\geq 2, \label{Lk0} \\
\| \Psi_\beta \|_{H^1 (Y_F)} \leq C \beta^{-1/4} . \label{LH10}
\end{gather}
\end{lemma}

\begin{proof}
First, in the variational formulation (\ref{VV}) we use the test function
$\varphi=\Psi_\beta$. It yields
\begin{equation}
\label{Energ1}
\int_{Y_F} | \nabla \Psi_\beta |^2 \, dy + \beta  \int_{Y_F} \Phi (\Psi_\beta ) \Psi_\beta \, dy
= - \int_S \sigma\Psi_\beta \, dS .
\end{equation}
Since $\beta\geq1$ and $\Phi(x) x \geq H^2 \, |x|^2$ because of (\ref{Ineqsign}), we deduce
from (\ref{Energ1}) that $\| \Psi_\beta \|_{H^1(Y_F)} \leq C$ (this estimate will be improved later).
Second, in the variational formulation (\ref{VV}) we use a test function which
is a (monotone) regularization of the sign of $\Psi_\beta$. The first term
in (\ref{VV}) is non-negative and the right hand side is bounded thanks to the
previous estimate. Thus, after passing to the regularization
parameter limit, we get
$$
\beta \int_{Y_F} \Phi (\Psi_\beta ) \mbox{ sign} (\Psi_\beta) \, dy \leq C,
$$
and after applying inequality (\ref{Ineqsign}) we get (\ref{L10}).
Next, we consider again (\ref{Energ1}) where the nonlinear term
is bounded from below using inequality (\ref{Ineqderiv})
\begin{equation}
\label{Secineq}
\beta \int_{Y_F} \Phi(\Psi_\beta) \Psi_\beta \, dy \geq \beta
\left( C_0 \int_{Y_F} |\Psi_\beta|^2 \, dy +
C_k \int_{Y_F} |\Psi_\beta|^k \, dy \right) .
\end{equation}
Furthermore, using a trace inequality \cite{ladyz}, we get
\begin{gather}
\left| \int_S \sigma\Psi_\beta \, dS \right| \leq C \| \Psi_\beta \|_{L^2(S)}
\leq C \| \Psi_\beta \|_{H^1(Y_F)}^{1/2} \| \Psi_\beta \|_{L^2 (Y_F)}^{1/2} \notag \\
\leq C \left( \beta\delta \| \Psi_\beta \|_{L^2 (Y_F)}^{2} + (\beta\delta)^{-1/3} \| \Psi_\beta \|_{H^1(Y_F)}^{2/3} \right) , \label{Estim}
\end{gather}
where we used Young's inequality $ab\leq a^4/4 + 3b^{4/3}/4$ for
$a=(\beta\delta)^{1/4}\| \Psi_\beta \|_{L^2 (Y_F)}^{1/2}$ and
$b=(\beta\delta)^{-1/4}\| \Psi_\beta \|_{H^1(Y_F)}^{1/2}$.
For $\beta \geq 1$ and $\delta>0$ small enough, (\ref{Lk0})-(\ref{LH10}) is a direct consequence of (\ref{Estim}).
\end{proof}

Since $S$ is compact there exist finitely many points $y_i^0 \in S$ and
neighborhoods $\mathcal{N}(y_i^0)$, $1\leq i\leq M$, such that the open sets
$W_i = Y_F \cap \mathcal{N}(y_i^0)$ cover $S$, i.e.,
$S \subset \cup^M_{i=1} \overline{W}_i$.
Take $W_0 \subset \subset Y_F$ so that ${ Y}_F \subset \cup^M_{i=0} W_i$ and let
$\{ \zeta_i \}^M_{i=0}$ be an associated partition of unity. Here $Y_F$ and $S$
are considered as subsets of the unit torus $\mathbb T^d$, so the functions
$\zeta_i(y)$ are 1-periodic.

\begin{proposition}
In each set $q(W_i)$ (the image of $W_i$ by the map $y\to q$)
define a boundary layer function
\begin{equation}
\label{BLf}
\psi^{bl}_i(q) = \beta^{-1/2} \Psi_0(q',\beta^{1/2} q_d) =
\frac{-\sigma(q')}{\sqrt{\beta \Phi^\prime(0)} }
\exp \{ -q_d \sqrt{\beta \Phi^\prime(0)} \} ,
\end{equation}
where $\Psi_0$ is defined by (\ref{Leading1}).
For any smooth test function $\varphi$, such that $\varphi =0$ on
$\p q(W_i) \setminus q(S\cap W_i)$, it satisfies
\begin{gather}
\int_{q(W_i)} K^i \nabla_q \psi^{bl}_i \cdot \nabla_q \varphi \, J^i \, dq +
\beta \int_{q(W_i)} \Phi (\psi^{bl}_i) \varphi \, J^i \, dq  \notag \\
+ \int_{q(S\cap W_i)} \sigma \varphi \sqrt{1+ | \nabla_{q'} \gamma_i |^2} \, d q' =
\int_{q(W_i)} R^i_\beta \varphi \, dq ,
\label{Varbl1}
\end{gather}
where $J^i$ is the Jacobian of the map $y\to q$, defined by (\ref{jacobian}),
$K^i$ is the metric matrix defined by (\ref{metric})
and
\begin{equation}
\label{Rest1}
\| R^i_\beta \|_{L^\infty(q(W_i))} \leq C , \quad
\| R^i_\beta \|_{L^1(q(W_i))} \leq \frac{C}{\sqrt{\beta}}  .
\end{equation}
\end{proposition}

\begin{proof}
By direct calculations, using the explicit formula (\ref{BLf}) and taking
into account that $K^i_{k,d}=0$ for $1\leq k \leq d-1$.
\end{proof}

Considering the boundary layers $\psi^{bl}_i$ as functions of $y$ now,
we immediately obtain the following corollary.

\begin{corollary}
The boundary layers $\psi^{bl}_i(q(y))$ satisfy, for any $\varphi \in H^1 (W_i)$,
\begin{gather}
\int_{W_i} \nabla_y (\zeta_i \psi^{bl}_i) \cdot \nabla_y \varphi \, dy +
\beta \int_{W_i} \Phi ( \zeta_i \psi^{bl}_i ) \varphi \, dy +
\int_{S\cap W_i} \zeta_i \sigma \varphi \, dS = \notag \\
\int_{W_i } R^i_\beta \varphi \ dy ,
\label{Varbl2}
\end{gather}
where the redefined reminders $R^i_\beta$ satisfy
\begin{equation}
\label{Rest2}
\| R^i_\beta \|_{L^\infty(W_i)} \leq C , \quad
\| R^i_\beta \|_{L^1(W_i)} \leq \frac{C}{\sqrt{\beta}}  .
\end{equation}
\end{corollary}

Due to the geometric assumptions from Section \ref{Geo}, for $\mu$ sufficiently small,
the tubular neighborhood $Y_F^\mu = \{ y\in Y_F \ : \, \mbox{dist} (y,S) < \mu \}$
satisfies $Y_F^\mu \subset \cup^M_{i=1} W_i$.

An arbitrary smooth function $f$, defined on $Y_F$ is then written as
$f=\sum_{i=1}^M \zeta_i f + \zeta_0 f$. The boundary between $Y_F^\mu$
and $Y_F \setminus {\overline Y}_F^\mu$ is $C^3$ and we extend smoothly
$f$ from $Y_F$ into $Y_F \setminus {\overline Y}_F^\mu$.
In $Y_F \setminus {\overline Y}_F^\mu$, $f$, together with its derivatives,
is exponentially small with respect to $1/\beta$.

In the calculations which follow we replace $f$ by the above extension
of $\sum_{i=1}^M \zeta_i f$ from $Y_F^\mu$ to $Y_F$. The error is
exponentially small in $1/\sqrt{\beta}$ and we ignore it.

Collecting together the boundary layers with the associated partition of unity,
we define
\begin{equation}
\label{bl}
\psi^{bl} = \sum_{i=1}^M \zeta_i  \psi^{bl}_i
\end{equation}
and deduce the following result.

\begin{proposition}
For any $\varphi  \in H^1_\#(Y_F)$, we have
\begin{gather}
\int_{Y_F} \nabla (\Psi_\beta - \psi^{bl} ) \cdot \nabla \varphi \, dy
+ \beta \int_{Y_F} \left(\Phi (\Psi_\beta) - \Phi (\psi^{bl}) \right) \varphi \, dy
= \int_{Y_F } R_\beta \varphi \, dy ,
\label{Varbl3}
\end{gather}
where the global reminder $R_\beta$ satisfies
\begin{equation}
\label{Rest3}
\| R_\beta \|_{L^\infty(Y_F)} \leq C , \quad
\| R_\beta \|_{L^1(Y_F)} \leq \frac{C}{\sqrt{\beta}}  .
\end{equation}
\end{proposition}

\begin{proof}
We obtain (\ref{Varbl3}) by subtracting the variational formulations (\ref{Varbl2})
of the boundary layers $\psi^{bl}_i$ from the variational formulation (\ref{VV})
of $\Psi_\beta$. We use the fact that
$$
\Phi (\sum_{i=1}^M \zeta_i \psi^{bl}_i ) = \sum_{i=1}^M \zeta_i \Phi (\psi^{bl}_i)
+ O(\beta^{-1/2})
$$
since $\psi^{bl}_i=O(\beta^{-1/2})$ and $\Phi (0)=0$.
\end{proof}

Finally we obtain the main result of this section which is a rigorous justification
of (\ref{Leading2}) (recall that $\psi^{bl}_i(q) = \beta^{-1/2} \Psi_0(q',\beta^{1/2} q_d)$).

\begin{theorem}
\label{estlargebet}
Let $\Psi_\beta$ be the unique solution of (\ref{BP1}) and $\psi^{bl}$ be given
by (\ref{bl}).
There exists a positive constant $C$ such that, $\forall\,\beta\geq1$,
\begin{gather}
\| \Psi_\beta - \psi^{bl} \|_{L^1 (Y_F)} \leq \frac{C}{\beta^{3/2}} , \label{L1M0} \\
\| \Psi_\beta - \psi^{bl} \|_{L^2 (Y_F)} \leq \frac{C}{\beta^{5/4}} ,\label{L20} \\
\| \Psi_\beta - \psi^{bl} \|_{H^1 (Y_F)} \leq \frac{C}{ \beta^{3/4}} . \label{LH1M0}
\end{gather}
\end{theorem}

\begin{proof}
The proof is similar to that of Lemma \ref{thm.estim}.
First, we test (\ref{Varbl3}) by the regularized sign of
$\Psi_\beta - \psi^{bl}$.
After passing to the regularization parameter limit and using
the second inequality of (\ref{Rest3}), we get
$$
\beta \int_{Y_F} \left(\Phi (\Psi_\beta) - \Phi (\psi^{bl})\right)
\mbox{ sign} \left(\Psi_\beta - \psi^{bl} \right) \, dy
\leq \frac{C}{\sqrt{\beta}},
$$
Since (\ref{Ineqderiv}) implies that $\Phi^\prime(x)\geq C>0$, we deduce (\ref{L1M0}).

Next we test (\ref{Varbl3}) by $\Psi_\beta - \psi^{bl}$.
It yields
\begin{gather}
\int_{Y_F} |\nabla (\Psi_\beta - \psi^{bl})|^2 \, dy + \beta \int_{Y_F}
\left( \Phi (\Psi_\beta) - \Phi (\psi^{bl}) \right) (\Psi_\beta - \psi^{bl}) \, dy \notag \\
\leq \| R_\beta \|_{L^\infty(Y_F)} \| \Psi_\beta - \psi^{bl} \|_{L^1({ Y_F })}
\leq  \frac{C}{\beta^{3/2}}  .
\label{Energ2}
\end{gather}
For $\beta \geq 1$, (\ref{L20})-(\ref{LH1M0}) is a direct consequence of (\ref{Energ2}).
\end{proof}


Theorem \ref{estlargebet} justifies the approximation (\ref{Leading2})
by providing error estimates in integral norms. The next result gives pointwise
estimates for the same asymptotic approximation.

\begin{lemma}
\label{l_1}
There exist positive constants $\beta_0$, $C_1$ and $C_2$ such that,
for all $\beta>\beta_0$ and for all $y\in Y_F$, the following estimates hold:
\begin{equation}\label{psi_bou}
|\Psi_\beta(y)|\le \frac{C_1}{\sqrt{\beta}}\exp\big\{-C_2\sqrt{\beta}\,
\mathrm{dist}(y,S)\big\},
\end{equation}
\begin{equation}\label{psi_bou_d}
|\nabla\Psi_\beta(y)|\le C_1\exp\big\{-C_2\sqrt{\beta}\,
\mathrm{dist}(y,S)\big\},
\end{equation}
\begin{equation}\label{psi0_bou}
|\Psi_\beta(y)- \psi^{bl}(y)|\le \frac{C_1}{\beta}\exp\big\{-C_2\sqrt{\beta}\,
\mathrm{dist}(y,S)\big\}.
\end{equation}
\end{lemma}

\begin{proof}
Introduce a function of $s\in\mathbb R$
$$
p(s) = \left\{\begin{array}{ll}
\Phi(s)/s \qquad  \hbox{for }s\not=0,\\[2mm]
\dsp \sum_{j=1}^N z_j^2n_j^0(\infty) \qquad  \hbox{for }s=0.
\end{array}
\right.
$$
From (\ref{Ineqsign}) in Lemma \ref{Ineq1}, we deduce $p(s)\geq H^2>0$ for all
$s\in\mathbb R$. It also readily follows from the definition of $p$ that $p$
is a continuous function of $s$.

Recall that $Y_F$ and $S$ are subsets of the unit torus $\mathbb T^d$.
Therefore, 1-periodic boundary conditions are implicit for all boundary
value problems below. For the sake of brevity we do not indicate them.
Introducing a function $B_\beta(y)=p(\Psi_\beta(y))$ (which is continuous
and satisfies $B_\beta(y)\geq H^2$), the Poisson-Boltzmann equation (\ref{BP1})
can be rewritten as
$$
\left\{
\begin{array}{ll}
-\Delta\Psi_\beta+\beta B_\beta(y)\Psi_\beta=0,\qquad&\hbox{in }Y_F,\\[2mm]
\nabla\Psi_\beta\cdot{\bf n}=-\sigma, \qquad&\hbox{on }S.
\end{array}
\right.
$$
Denote $\Sigma=\|\sigma\|_{L^\infty(S)}$. Then, by the maximum principle,
$$
|\Psi_\beta|\le \Psi_\beta^+ \qquad\hbox{in }Y_F,
$$
where $\Psi_\beta^+$ is the unique solution of
$$
\left\{
\begin{array}{ll}
-\Delta\Psi_\beta^+ +\beta H^2\Psi_\beta^+=0,\qquad&\hbox{in }Y_F,\\[2mm]
\nabla\Psi_\beta\cdot{\bf n}=\Sigma, \qquad&\hbox{on }S.
\end{array}
\right.
$$
Thus, it suffices to show that
\begin{equation}
\label{main_ineq}
|\Psi^+_\beta(y)|\le \frac{C_1}{\sqrt{\beta}}\exp\big\{-C_2\sqrt{\beta}\,
\mathrm{dist}(y,S)\big\},\qquad \hbox{\rm for all }y\in Y_F.
\end{equation}

In order to prove (\ref{main_ineq}), we are going to construct a so-called
barrier function. For any $y\in S$ denote by $R(y)$ the radius of curvature of $S$
at $y$. Under our standing assumption on the smoothness of $S$ we have
$R(y)\geq R_0>0$, $\forall\,y\in S$. For $R_0$ small enough, each of
the neighborhoods $\mathcal{N}(y_i^0)$, $1\leq i\leq M$, covering $S$
contains a ball of center $y_i^0$ and radius $(R_0/2)$. In this ball,
we rewrite the Laplace operator in terms of the new coordinates $q=q(y)$,
introduced in Subsection \ref{Geo},
$$
-\Delta_y=  -\frac{\partial^2}{\partial q_d^2}+\sum\limits_{i,j=1}^{d-1}
Q_{ij}(q)\frac{\partial^2}{\partial q_i\partial q_j}+\sum\limits_{j=1}^{d}
Q^0_{j}(q)\frac{\partial}{\partial q_j}  ,
$$
with regular bounded functions $Q_{ij}(q)$ and $Q_j^0(q)$ defined in
terms of $K_{ij}$ and $J$.
Setting
$$
\mathcal{G}(s)=\left\{
\begin{array}{ll}
s-\frac{s^2}{R_0}\qquad&\hbox{if }s\le \frac{R_0}{2},\\[2mm]
\frac{R_0}{4},&\hbox{otherwise,}
\end{array}\right.
$$
and
$$
U(y )= \frac{2\Sigma}{H\sqrt{\beta}} \exp \left\{ -\sqrt{\beta}\frac{H}{2}
\mathcal{G}(q_d)\right\} ,
$$
after straightforward computations we obtain
\begin{equation}
\label{b_c_aux}
\nabla U \cdot {\bf n}\Big|_S = - \frac{\partial U}{\partial q_d}\Big|_{q_d=0}=\Sigma.
\end{equation}
Notice also that  $U\in H^2(Y_F)$ because $\nabla U(y)=0$ if $\mathrm{dist}(y,S)=R_0/2$.

Substituting $U$ in the equation yields, for all $y\in Y_F$ such that $q_d\le R_0/2$,
$$
-\Delta U+\beta H^2 U= \beta\Big\{\!-\frac{H^2}{4}\Big(1-\frac{2q_d}{R_0}\Big)^2
-\frac{H}{R_0\sqrt{\beta}}-\frac{Q_d^0(q)H}{2\sqrt{\beta}}
\Big(1-\frac{2q_d}{R_0}\Big)+H^2\Big\}U .
$$
Clearly, for all sufficiently large $\beta$ the above right hand side is positive.
Therefore, for $\beta>\beta_0$,
$$
-\Delta U+\beta H^2 U\ge 0 \qquad\hbox{in }Y_F.
$$
Combining this relation with (\ref{b_c_aux}), and using the maximum principle,
we conclude that $\Psi^+_\beta\le U$ in $Y_F$. Since $\mathrm{dist}(y,S)$ in $Y_F$
is bounded by some constant $C_3$ and $\mathcal{G}(s)\ge C_4 s$ for any $s\in(0,C_3)$,
this implies the first desired estimate (\ref{psi_bou}).

Estimate (\ref{psi_bou_d}) follows from (\ref{psi_bou}) thanks
to the standard elliptic estimates. Indeed, in the rescaled coordinates
$z=\sqrt{\beta}y$ equation (\ref{BP1}) reads
$$
\begin{array}{c}
\displaystyle
-\Delta_z\Psi_\beta=-\Phi(\Psi_\beta) \qquad \hbox{in }\sqrt{\beta}Y_F\\[2mm]
\displaystyle
\nabla_z\Psi_\beta\cdot n_z=-\frac{1}{\sqrt{\beta}}\sigma\qquad \hbox{on } \sqrt{\beta}S.
\end{array}
$$
By (\ref{psi_bou}) for any ball $B_{z_0,1}=\{z\,:\,|z-z_0|\le 1\}$ with $y_0\in \sqrt{\beta}Y_F$ we have
$$
|\Psi_\beta(z)|\le \frac{C}{\sqrt{\beta}}e^{-c_2 \mathrm{dist}(z_0,\sqrt{\beta}S)},\quad |\Phi(\Psi_\beta(z))|\le \frac{C}{\sqrt{\beta}}e^{-c_2 \mathrm{dist}(z_0,\sqrt{\beta}S)},
$$
$z\in B_{z_0,1}\cap\sqrt{\beta}Y_F$.
Considering our regularity assumptions on $S$ and $\sigma$, by the local elliptic estimates for Poisson equation including those near the boundary, we obtain
$$
|\nabla_z \Psi_\beta(z)|\le  \frac{C}{\sqrt{\beta}}e^{-c_2 \mathrm{dist}(z_0,\sqrt{\beta}S)},\qquad z\in B_{z_0,1/2}\cap\sqrt{\beta}Y_F.
$$
In the coordinates $y$ this yields the desired estimate (\ref{psi_bou_d}).

Estimate (\ref{psi0_bou}) can be obtained by means of similar arguments
as in the proof of (\ref{psi_bou}). Here we just outline  the proof and leave the details to the reader.
From the definition of $\Psi_\beta$ and $\psi^{bl}$ and estimate (\ref{psi_bou}) it readily follows that the difference $V_\beta=\Psi_\beta-\psi^{bl}$ satisfies in $Y_F^\mu$, $\mu=R_0/2$, the following problem:
$$
\begin{array}{c}
-\Delta V_\beta(y)+ \beta\Phi'(0)V_\beta(y)=g_1(y)\qquad \hbox{in }Y_F^\mu,\\[2mm]
\nabla V_\beta\cdot {\bf n}\big|_S=0,\qquad V_\beta\big|_{|q_d(y)|=\frac{R_0}{2}}=g_2(y)
\end{array}
$$
with
$$
|g_1(y)|\le c_1 e^{-c_2\sqrt{\beta}q_d(y)},\qquad
|g_2(y)|\le \frac{1}{\sqrt{\beta}}c_1e^{-c_2\sqrt{\beta}R_0/2}\le \frac{1}{\beta}c_1e^{-c_3\sqrt{\beta}R_0/2};
$$
here the constants $c_2>0$ and $c_3>0$ do not depend on $\beta$. Setting $\overline V_\beta=\frac{C}{\beta}e^{-\sqrt{\beta}H_1q_d(y)}$
and choosing large enough $C>0$ and small enough $H_1>0$,
we obtain that for all sufficiently large $\beta$ it holds
$$
-\Delta\overline V_\beta+\beta\Phi'(0)\overline V_\beta> g_1,
\quad\nabla\overline V_\beta\cdot{\bf n}\big|_S<0,\quad
\overline V_\beta\big|_{|q_d(y)|=\frac{R_0}{2}}> g_2.
$$
Therefore, $V_\beta\leq \overline V_\beta$ in $Y_F^\mu$. Similarly, $V_\beta\geq -\overline V_\beta$ in $Y_F^\mu$. This yields (\ref{psi0_bou}) in $Y_F^\mu$.  In $Y_F\setminus Y_F^\mu$ (\ref{psi0_bou}) follows by the maximum principle.
\end{proof}

\section{Dirichlet boundary condition or $\zeta$ potential at the boundary}
\label{dirichlet}

The previous asymptotic analysis was specific to the Neumann boundary
condition (or given charge density $\sigma$) imposed on the pore walls $S$.
The situations is quite different for Dirichlet boundary condition (or
$\zeta$ potential) on $S$. We briefly investigate this case. We modify
the boundary condition in the Poisson-Boltzmann equation
\begin{equation}
\label{BP1Z}
\left\{ \begin{array}{l}
\dsp - \Delta \Psi_\beta + \beta  \Phi ( \Psi_\beta  ) =0  \ \mbox{ in } \ Y_F , \\
\dsp   \Psi_\beta  = \zeta \ \mbox{ on } \, S , \\
 \Psi_\beta \; \mbox{ is } 1-\mbox{periodic}.
\end{array} \right.
\end{equation}
All unknowns and parameters are exactly the same as the ones in Section \ref{PBE}.

\subsection{The limit case of small $\beta$}

We start by studying the behavior of $\Psi_\beta$ when $\beta$ goes to zero.
Performing a formal asymptotic expansion
$$
\Psi_\beta = \Psi_0 + \beta \Psi_1 + \dots ,
$$
it is easy to check that the zero-order term is constant
$$
\Psi_0(y)  \equiv \zeta ,
$$
while the first-order term is the solution of the linear problem
\begin{equation}
\label{BP2Z}
\left\{ \begin{array}{l}
\dsp - \Delta \Psi_1 = - \Phi(\zeta)  \ \mbox{ in } \ Y_F , \\
\dsp   \Psi_1  = 0 \ \mbox{ on } \, S , \\
 \Psi_1 \; \mbox{ is } 1-\mbox{periodic}.
\end{array} \right.
\end{equation}
It is not difficult to justify this ansatz and to prove the following
error estimate.

\begin{lemma}
\label{Dir.small}
There exists a positive constant $C$ such that
$$
\| \Psi_\beta - \zeta - \beta \Psi_1 \|_{H^1(Y_F)} \leq C \beta^2 .
$$
\end{lemma}

\subsection{The limit case of large $\beta$: formal asymptotics}
\label{Form1}

We are now interested in the behavior of $\Psi_\beta$ for large $\beta$.
As in subsection \ref{Form} we are going to use matching asymptotic expansion.
Of course, the outer expansion, being independent of the boundary condition,
is the same and we get again (\ref{Outerexp}), namely
\begin{equation}
\label{Outerexp1}
\Psi_\beta (y) = O (\frac{1}{\beta} ) \quad \mbox{away from the boundary }  S.
\end{equation}
To the contrary, the behavior close to $S$ differs significantly from Section \ref{Lb}.
We study it locally, in the same geometrical setting as before.
We obtain the same differential operator close to the boundary with
a different boundary condition
\begin{gather}
-\frac{\p^2 \Psi_\beta}{\p q_d^2} + \beta \Phi(\Psi_\beta)
+ \mbox{ lower order derivatives in } q_d \notag \\
+ \mbox{ second order differential operator in } q' \, = \, 0
\mbox{ in } \ Y_F \cap \mathcal{N}(y_0) , \label{BP1flat1Z} \\
\Psi_\beta = \zeta (q_1 , \dots , q_{d-1} )  \ \mbox{ on } \, S\cap \mathcal{N}(y_0) .\label{BP1flat2Z}
\end{gather}
The boundary condition (\ref{BP1flat2Z}) (given $\zeta$ potential)
is a much stronger constraint than the previous Neumann condition (\ref{BP1flat2}).
Consequently, we change the inner asymptotic expansion which, instead of (\ref{expan}),
is now of the form
\begin{gather}
\Psi_\beta(q',q_d)=
\Psi_{0, \zeta} (q',\beta^{1/2} q_d)+\beta^{-1/2}\Psi_{1,\zeta} (q',\beta^{1/2} q_d)+\dots .
\label{expan1}
\end{gather}
Introducing $\xi=\beta^{1/2} q_d$, substituting (\ref{expan1}) in
(\ref{BP1flat1Z})-(\ref{BP1flat2Z}) and collecting power-like terms
in the resulting equations, we arrive at the following problem for
the leading term of the expansion
\begin{gather}
-\frac{d^2}{d\xi^2}\Psi_{0, \zeta } (q',\xi) + \Phi (\Psi_{0, \zeta} (q',\xi))=0 \; \mbox{for} \;  \xi>0 , \label{Inner3} \\
\Psi_{0, \zeta} (q',0)=\zeta(q').\label{Inner4}
\end{gather}
After matching with the outer solution $\Psi_\beta=O (\beta^{-1})$, we impose
additionally that $\Psi_{0,\zeta} (q',+\infty)=0$ and the square integrability
of the derivative.

Let $\Cc(x) = \dsp \sum^N_{j=1} n^0_j(\infty) e^{-z_j x}$ be the primitive of $\Phi (x)$.
Then the problem (\ref{Inner3})-(\ref{Inner4}) admits the first integral
\begin{equation}
\label{Firstint}
- \frac{1}{2} (\frac{d}{d\xi} \Psi_{0,\zeta} )^2 + \Cc (\Psi_{0, \zeta } )= C_1 = \mbox{ constant}.
\end{equation}
As we impose $\dsp \Psi_{0, \zeta } (q',+\infty)=0$ and the square integrability
of the derivative, it follows that
\begin{equation}
\label{Cte}
C_1 = \Cc (0) = \sum^N_{j=1} n^0_j(\infty) > 0.
\end{equation}
Thus, (\ref{Inner4}), (\ref{Firstint}) and (\ref{Cte}) give
\begin{equation}
\label{Cauchy}
\left\{ \begin{array}{l}
\dsp \Psi_{0, \zeta } |_{\xi =0} =\zeta ,  \\
\dsp \frac{d}{d\xi} \Psi_{0,\zeta} = -2\,\mbox{sign} (\zeta)
\sqrt{ \Cc (\Psi_{0,\zeta} ) - \Cc (0)}.
\end{array} \right.
\end{equation}

\begin{proposition}
\label{Bdrylay}
The Cauchy problem (\ref{Cauchy}) has a unique smooth solution $\Psi_{0,\zeta}$
on $(0, +\infty )$, satisfying (\ref{Inner3}) and
\begin{equation}
\label{Decay}
\begin{array}{l}
\dsp | \Psi_{0,\zeta} (q',\xi) | \leq |\zeta (q') | e^{-\sqrt{C_s} \xi}, \\
\dsp | \frac{d}{d\xi} \Psi_{0 \zeta} (q',\xi)| \leq \sqrt{C_0} |\zeta (q') |^{1/2}
e^{-\sqrt{C_s }\xi /2 },
\end{array}
\end{equation}
where $C_s = \min \{ \sum_{j\in j^-} z^2_j n_j^0 (\infty) , \sum_{j\in j^+} z^2_j n_j^0 (\infty) \} $ and $\dsp C_0 = 2 \max_S | \Phi (\zeta (q')  )| $.
\end{proposition}

\begin{proof}
For $\zeta =0$, the unique solution is $\Psi_{0, \zeta } =0$ and there is nothing to prove.

For $\zeta \neq 0$, the Cauchy problem (\ref{Cauchy}) has a unique maximal smooth solution on some interval $(0,\ell )$. If $\zeta >0$, then the solution is positive, monotone decreasing and it reaches the value $0$ at $\xi =\ell$. For $\zeta <0$, it is the opposite situation. But $0$ is a critical point of (\ref{Cauchy}) and no trajectory can leave or reach that point. So the solution cannot be zero for some finite $\ell$.
Therefore, the Cauchy problem (\ref{Cauchy}) has a unique maximal smooth solution on the entire real line $(0, +\infty )$.

Next, a simple calculation gives
$$
| \frac{d}{d\xi } \Psi_{0, \zeta } (q' , \xi )|^2 \geq C_s |\Psi_{0, \zeta } (q' , \xi )|^2 ,
$$
with $C_s = \min \{ \sum_{j\in j^-} z^2_j n_j^0 (\infty) , \sum_{j\in j^+} z^2_j n_j^0 (\infty) \} $.
Consequently, for $\zeta >0$, we have
$$ \frac{d}{d\xi } \Psi_{0, \zeta } (q' , \xi ) \leq \zeta (q') e^{-\sqrt{C_s} \xi} $$
and we establish the exponential decay of $ \Psi_{0, \zeta }$ and the first part of (\ref{Decay}). For $\zeta <0$ everything is analogous.

The ordinary differential equation (\ref{Cauchy}) gives
$$
| \frac{d}{d\xi } \Psi_{0, \zeta } (q' , \xi )|^2 \leq 2 \max_S | \Phi (\zeta ) |  |\Psi_{0, \zeta } (q' , \xi )|
$$
and we conclude the remaining part of  estimate (\ref{Decay}).
\end{proof}

\begin{remark}
In many situations, for a symmetric electrolyte with ion charges $\pm Q$,
the explicit solutions are known. A classical reference is the book \cite{Verwey}.
For example, in the case $-z_1 =1=z_2$ and $n^0_1(\infty) = n^0_2(\infty)=1/2$,
we have the following Gouy-Chapman solution
$$
\Psi_{0, \zeta } ( q', \xi ) = 2\ln \frac{1+ \tanh (\zeta /2) e^{-\xi} }{1- \tanh (\zeta /2) e^{-\xi} } .
$$
By direct computation we can check that this solution satisfies the properties
established in Proposition \ref{Bdrylay}. In the general case $\Psi_{0, \zeta }$
can be expressed using elliptic functions. Nevertheless, our simple analysis
gave us the properties of the solution without using its explicit form.
\end{remark}

\subsection{The limit case of large $\beta$: rigorous error estimate}
\label{REE2}

We start with the study of the boundary layer function $\Psi_{0, \zeta }$.
As in section \ref{Lb} we use the local change  of variables $y\to q$.

\begin{proposition}
For each open subset $W_i = Y_F \cap {\cal N}(y_i^0)$, let the boundary layer function
be defined by
\begin{equation}
\label{BLfZ}
\psi^{bl}_i (q ) = \Psi_{0, \zeta } (q', q_d \sqrt{\beta})
\end{equation}
for the given boundary data $\zeta$ on $S\cap {\overline W}_i$.
Then, for any smooth function $\varphi$ such that $\varphi =0$ on
$\p \mathbf{q}(S \cap {\overline W}_i)$, we have
\begin{gather}
\hskip-10pt\int_{\mathbf{q}(W_i)} K^i \nabla_q \psi^{bl}_i (q ) \cdot \nabla_q \varphi \ J^i \ dq +
\beta \int_{\mathbf{q}(W_i)} \Phi ( \psi^{bl}_i (q ) ) \, \varphi J^i \ dq
= \int_{\mathbf{q}(W_i)} R^i_1 \nabla_{{ q'}} \varphi \, J^i \, dq \notag \\
+\int_{\mathbf{q}(W_i)} R^i_2  \varphi \, J^i \, dq  = { \int_{\mathbf{q}(W_i)} R^i_3  \varphi \, J^i \, dq  + \int_{\mathbf{q}(\p W_i \setminus (S\cap {\overline W}_i))} \sigma^i_1  \varphi  \, dq',}
\label{Varbl1Z}
\end{gather}
where $J^i$ is the Jacobian of the map $q \to y$, defined by (\ref{jacobian}),
$K^i$ is the metric matrix defined by (\ref{metric}), $\nabla_{q'}$ is gradient
with respect to $q'$ and
\begin{gather}
\hskip-16pt \sqrt{\beta} ( \|  R^i_1 \|_{L^\infty ( {q} (W_i))}  + \|  \sigma^i_1 \|_{L^\infty ( {q} (\p W_i \setminus (S\cap {\overline W}_i)))} )+ \|  R^i_j \|_{L^\infty ( {q} (W_i))} \leq C\sqrt{\beta} ; \label{Rest1Z}\\
\sqrt{\beta} (\|  R^i_1 \|_{L^r ( {q} (W_i))}+ \|  \sigma^i_1 \|_{L^\infty ( {q} (\p W_i \setminus (S\cap {\overline W}_i)))} ) + \notag \\
\|  R^i_j \|_{L^r ( {q} (W_i))} \leq C\beta^{(1-1/r)/2} \; j=2,3 . \label{Rest2Z}
\end{gather}
\end{proposition}

\begin{proof}
By direct calculation, using the estimate (\ref{Decay}). We note that the higher derivatives of $\Psi_{0, \zeta }$ with respect to $q'$ satisfy also the estimate (\ref{Decay}).
\end{proof}

\begin{corollary}
For all $\varphi \in H^1(W_i)$ such that $\varphi =0$ on $S\cap \p W_i$,
we have
\begin{gather}
\int_{W_i} \nabla_y \psi^{bl}_i \cdot \nabla_y \varphi \, dy +
\beta \int_{W_i} \Phi (\psi^{bl}_i) \varphi \, dy
= \int_{W_i} R^i_1 \nabla_y \varphi \, dy \notag \\
+ \int_{W_i} R^i_2 \varphi \, dy = \int_{W_i} R^i_3  \varphi \, dy
+ \int_{\p W_i \setminus (S\cap {\overline W}_i)} \sigma^i_1  \varphi  \, dS,
\label{Varbl2Z}
\end{gather}
where redefined reminders $\sigma^i_1$, $R^i_j$, $j=1,2,3$, satisfy (\ref{Rest1Z})-(\ref{Rest2Z}).
\end{corollary}

Using the definition of $\psi^{bl}_i$, we see that $\psi^{bl}_i = \psi^{bl}_j$
on $W_i \cap W_j$, when $W_i \cap W_j$ is nonempty. Now let $\psi^{bl} =\psi^{bl}_i$ on $W_i$.
As in Subsection \ref{REE}, we make a smooth extension of $\psi^{bl}$ into
$Y_F \setminus {\overline Y}_F^\mu$. In $Y_F \setminus {\overline Y}_F^\mu $,
$\psi^{bl}$, together with its derivatives, is exponentially small with respect
to $1/ \beta$. We simply ignore the exponentially small terms in our estimates.
We have

\begin{proposition}
For any $\varphi \in H^1(Y_F)$ such that $\varphi =0$ on $S$, we have
\begin{gather}
\int_{Y_F}  \nabla (\Psi_\beta -  \psi^{bl}) \cdot \nabla \varphi \ dy + \beta  \int_{Y_F} (\Phi (\Psi_\beta ) - \Phi ( \psi^{bl}) ) \varphi \ dy =\int_{Y_F } R^1_\beta \nabla \varphi \ dy \notag \\
+ \int_{Y_F } R^2_\beta \varphi \ dy = \int_{Y_F} R^3_\beta  \varphi  \, dy
+ \sum_{i=1}^M\int_{\p W_i \setminus (S\cap {\overline W}_i)} \sigma^i_\beta  \varphi  \, dS,
\label{Varbl3Z}
\end{gather}
where the global reminders $R_\beta^j$, $j=1,2,3$ and $\sigma_\beta^i$, $i=1, \dots , M$
satisfy
\begin{gather}
\sqrt{\beta} (\| R_\beta^1 \|_{L^r(Y_F)} + \| \sigma_\beta^{i} \|_{L^r(\p W_i \setminus (S\cap {\overline W}_i))}) + \notag \\
 \| R_\beta^j \|_{L^r(Y_F)} \leq C\beta^{(1-1/r)/2} , \; \forall r\in [1, +\infty ], \,  j=2,3 , \, i=1, \dots, M .\label{Rest3Z}
\end{gather}
\end{proposition}

\begin{theorem}
\label{estlargebetZ}
Let $\Psi_\beta$ be given by (\ref{BP1Z}) and $\psi^{bl}$ by (\ref{BLfZ}).
Then we have the following behavior for large $\beta$:
\begin{gather}
\| \Psi_\beta - \psi^{bl} \|_{L^1 (Y_F)} \leq \frac{C}{\beta} , \label{L1M0Z} \\
\| \Psi_\beta -  \psi^{bl}\|_{L^2 (Y_F)} \leq \frac{C}{\beta^{1/2}} \label{L20Z} \\
\| \Psi_\beta -  \psi^{bl} \|_{H^1 (Y_F)} \leq C . \label{LH1M0Z}
\end{gather}
\end{theorem}

\begin{proof}
First we test the variant of (\ref{Varbl3Z}) involving boundary terms by the
regularized sign of $(\Psi_\beta - \psi^{bl})$.
After passing to the regularization parameter limit, we get
$$
\beta \int_{Y_F} (\Phi (\Psi_\beta ) - \Phi ( \psi^{bl}))
\mbox{ sign } (\Psi_\beta - \psi^{bl} ) \, dy \leq C,
$$
and after applying a slight generalization of the inequality (\ref{Ineqsign})
we get (\ref{L1M0Z}).

Next we use the variant of (\ref{Varbl3Z}) involving only volume terms and
test (\ref{Varbl3Z}) by $g_\beta =\Psi_\beta - \psi^{bl}$. It yields
\begin{gather}
\int_{Y_F} | \nabla( \Psi_\beta - \psi^{bl})|^2 \, dy
+ \beta \int_{Y_F} (\Phi (\Psi_\beta ) - \Phi ( \psi^{bl} ))
(\Psi_\beta - \psi^{bl}) \, dy \leq \notag \\
C\beta^{1/4} \max_i \| g_\beta \|_{L^2 (\p W_i \cap \mathcal{T})}
+ C \beta^{1/4} \| g_\beta \|_{L^2 (Y_F)} \notag \\
\leq \frac{\beta \min \Phi'}{2} \| g_\beta \|_{L^2 (Y_F)}^2
+ \frac{1}{2} \| \nabla g_\beta \|_{L^2 (Y_F)}^{2} +C .
\label{Energ2Z}
\end{gather}
For $\beta \geq \beta_0$, (\ref{L20Z})-(\ref{LH1M0Z}) is a direct consequence of (\ref{Energ2Z}).
\end{proof}
\begin{remark}
By the same arguments as in Section \ref{Lb}, one can obtain, in addition to the inequalities of Theorem~\ref{estlargebetZ}, pointwise estimates for $\Psi_\beta$ and for the discrepancy $\Psi_\beta- \psi^{bl}$. In the case of Dirichlet boundary condition these estimates read, for $\beta\ge 1$,
\begin{equation*}
|\Psi_\beta(y)|\le C_1\exp\big\{-C_2\sqrt{\beta}\,
\mathrm{dist}(y,S)\big\},
\end{equation*}
\begin{equation*}
|\nabla\Psi_\beta(y)|\le C_1\sqrt{\beta}\exp\big\{-C_2\sqrt{\beta}\,
\mathrm{dist}(y,S)\big\},
\end{equation*}
\begin{equation*}
|\Psi_\beta(y)- \psi^{bl}(y)|\le \frac{C_1}{\sqrt{\beta}}\exp\big\{-C_2\sqrt{\beta}\,
\mathrm{dist}(y,S)\big\},
\end{equation*}
where $C_1$ and $C_2$ are positive constants, independent of $\beta$. 
\end{remark}

\end{document}